\documentclass[12pt]{amsart}
\usepackage{mathrsfs}
\usepackage{amssymb}
\usepackage{verbatim}
\usepackage{hyperref}
\usepackage{color}

\begin{document}
\newtheorem*{thmC}{Theorem C}
\newtheorem*{thmA}{Theorem A}
\newtheorem*{thmB}{Theorem B}
\newtheorem{theorem}{Theorem}    
\newtheorem{proposition}[theorem]{Proposition}
\newtheorem{conjecture}[theorem]{Conjecture}
\def\theconjecture{\unskip}
\newtheorem{corollary}[theorem]{Corollary}
\newtheorem{lemma}[theorem]{Lemma}
\newtheorem{sublemma}[theorem]{Sublemma}
\newtheorem{observation}[theorem]{Observation}
\theoremstyle{definition}
\newtheorem{definition}{Definition}
\newtheorem{notation}[definition]{Notation}
\newtheorem{remark}[definition]{Remark}
\newtheorem{question}[definition]{Question}
\newtheorem{questions}[definition]{Questions}
\newtheorem{example}[definition]{Example}
\newtheorem{problem}[definition]{Problem}
\newtheorem{exercise}[definition]{Exercise}

\numberwithin{theorem}{section}
\numberwithin{definition}{section}
\numberwithin{equation}{section}

\def\earrow{{\mathbf e}}
\def\rarrow{{\mathbf r}}
\def\uarrow{{\mathbf u}}
\def\varrow{{\mathbf V}}
\def\tpar{T_{\rm par}}
\def\apar{A_{\rm par}}

\def\reals{{\mathbb R}}
\def\torus{{\mathbb T}}
\def\heis{{\mathbb H}}
\def\integers{{\mathbb Z}}
\def\naturals{{\mathbb N}}
\def\complex{{\mathbb C}\/}
\def\distance{\operatorname{distance}\,}
\def\support{\operatorname{support}\,}
\def\dist{\operatorname{dist}\,}
\def\Span{\operatorname{span}\,}
\def\degree{\operatorname{degree}\,}
\def\kernel{\operatorname{kernel}\,}
\def\dim{\operatorname{dim}\,}
\def\codim{\operatorname{codim}}
\def\trace{\operatorname{trace\,}}
\def\Span{\operatorname{span}\,}
\def\dimension{\operatorname{dimension}\,}
\def\codimension{\operatorname{codimension}\,}
\def\nullspace{\scriptk}
\def\kernel{\operatorname{Ker}}
\def\ZZ{ {\mathbb Z} }
\def\p{\partial}
\def\rp{{ ^{-1} }}
\def\Re{\operatorname{Re\,} }
\def\Im{\operatorname{Im\,} }
\def\ov{\overline}
\def\eps{\varepsilon}
\def\lt{L^2}
\def\diver{\operatorname{div}}
\def\curl{\operatorname{curl}}
\def\etta{\eta}
\newcommand{\norm}[1]{ \|  #1 \|}
\def\expect{\mathbb E}
\def\bull{$\bullet$\ }

\def\xone{x_1}
\def\xtwo{x_2}
\def\xq{x_2+x_1^2}
\newcommand{\abr}[1]{ \langle  #1 \rangle}

\newcommand{\Norm}[1]{ \left\|  #1 \right\| }
\newcommand{\set}[1]{ \left\{ #1 \right\} }
\def\one{\mathbf 1}
\def\whole{\mathbf V}
\newcommand{\modulo}[2]{[#1]_{#2}}

\def\scriptf{{\mathcal F}}
\def\scriptg{{\mathcal G}}
\def\scriptm{{\mathcal M}}
\def\scriptb{{\mathcal B}}
\def\scriptc{{\mathcal C}}
\def\scriptt{{\mathcal T}}
\def\scripti{{\mathcal I}}
\def\scripte{{\mathcal E}}
\def\scriptv{{\mathcal V}}
\def\scriptw{{\mathcal W}}
\def\scriptu{{\mathcal U}}
\def\scriptS{{\mathcal S}}
\def\scripta{{\mathcal A}}
\def\scriptr{{\mathcal R}}
\def\scripto{{\mathcal O}}
\def\scripth{{\mathcal H}}
\def\scriptd{{\mathcal D}}
\def\scriptl{{\mathcal L}}
\def\scriptn{{\mathcal N}}
\def\scriptp{{\mathcal P}}
\def\scriptk{{\mathcal K}}
\def\frakv{{\mathfrak V}}

\author{Zhengyang Li}
\address{Zhengyang Li
\\
School of Mathematical Sciences
\\
Beijing Normal University
\\
Laboratory of Mathematics and Complex Systems
\\
Ministry of Education
\\
Beijing 100875
\\
People's Republic of China
}
\email{zhengyli@mail.bnu.edu.cn}

\author{Qingying Xue}
\address{
        Qingying Xue\\
        School of Mathematical Sciences\\
        Beijing Normal University \\
        Laboratory of Mathematics and Complex Systems\\
        Ministry of Education\\
        Beijing 100875\\
        People's Republic of China}
\email{qyxue@bnu.edu.cn}
\thanks{The authors were supported partly by NSFC
(No. 11471041), the Fundamental Research Funds for the Central Universities (NO. 2014kJJCA10) and NCET-13-0065. \\ \indent Corresponding author: Qingying Xue\indent Email: qyxue@bnu.edu.cn}
\title
[Endpoint estimates for commutatorss of m-linear $\omega$-CZO]
{Endpoint estimates for the commutators of multilinear Calder\'{o}n-Zygmund operators with Dini type kernels}
\maketitle

\begin{abstract}
Let  $T_{\vec{b}}$  and $T_{\Pi b}$ be the  commutators in the $j$-th entry and  iterated commutators of the multilinear Calder\'{o}n-Zygmund operators, respectively. It was well-known that $T_{\vec{b}}$  and $T_{\Pi b}$ were not of weak type $(1,1)$ and $(H^1, L^1)$, but they did satisfy certain endpoint $L\log L$ type estimates. In this paper, our aim is to give more natural sharp endpoint results. We show that $T_{\vec{b}}$  and $T_{\Pi b}$ are bounded from product Hardy space $H^1\times\cdot\cdot\cdot\times H^1$ to weak $L^{\frac{1}{m},\infty}$ space, whenever the kernel satisfies a class of Dini type condition. This was done by using a key lemma given by M. Christ, a very complex decomposition of the integrand domains and splitting and estimating the commutators very carefully into several terms and cases.
\end{abstract}

\section{Introduction}
\subsection{Commutators of classical C-Z operators}

In 1976, Coifman, Rochberg and Weiss \cite{CRW} first introduced and studied the commutator of classical linear Calder\'{o}n-Zygmund singular integrals, which was defined by $$T_b = [b,T]f = bT(f)-T(bf).$$
The $L^p$ boundedness of $T_b$ was given in \cite{CRW} for $1<p<\infty$ when $b\in BMO(\reals^n)$. It is well-known that $T_b$ fails to be of weak type (1,1) and is not bounded from $H^1(\reals^n)$ to $L^1(\reals^n)$. Counterexamples were given by P\'{e}rez \cite{P} and Paluszy\'{n}ski \cite{Pa}. As an alternative result of the weak $(1,1)$ estimate of $T_b$, P\'{e}rez \cite{P} obtained the following $L(\log L)$ type endpoint estimate:
\begin{equation*}
|\{x\in\reals^n:|{T_bf(x)}|>\lambda\}|\leq C\int_{\reals^n}\frac{|f(x)|}{\lambda}\big(1+\log^+(\frac{|f(x)|}{\lambda})\big)dx, \quad \quad \quad\lambda>0.
\end{equation*}
Moreover, alternative results of the $(H^1, L^1)$ boundedness were also considered in the works of Alvarez \cite{A}, P\'{e}rez \cite{P}, and  Liang, Ky and Yang \cite{LKY}, which concerned with the boundedness of $T_b$ on the subspace of atomic Hardy Spaces, or concerned with the $(H_{w}^{1}, L_w^1)$ boundedness of $T_b$ if $b$ belongs to a subspace of $BMO$ which associated to a weight function $w$.

On the other hand, another more reasonable and alternative result of weak type $(1,1)$ and $(H^1, L^1)$ estimate was given by Liu and Lu \cite{LL} in 2002. The authors \cite{LL}  showed that $T_b$ is bounded from $H^1(\reals^n)$ to $L^{1,\infty}(\reals^n)$ if $b\in BMO(\reals^n)$. We note that, $T_b$ also fails to be bounded from $H^{p}(\reals^n)$ to $L^{p,\infty}(\reals^n)$ for $0<p<1$ by the generalized interpolation theorem (\cite[pp. 63]{Lu}). Therefore, the $(H^1, L^{1,\infty})$ boundedness of $T_b$ becomes a sharp endpoint estimate.
Moreover, it always holds that $L(\log L)(\mathbb{S}^{n-1})\subsetneq H^1(\mathbb{S}^{n-1})$ if $f$ vanishes on the unit sphere. However, there is no such inclusion relationship on $\reals^n$. Moreover, the inverse including relationship is still not true since the following example shows that $H^1(\reals^n)\nsubseteq L(\log L)(\reals^n)$.
\begin{example}
 Let \[f(x)=\frac{\chi _{[-\frac{1}{2}, \frac{1}{2}]}}{x\log_2^{1+\varepsilon}\frac{1}{|x|}}\ \ \text{for some $\varepsilon>0$},\]
\[a_j(x)=\frac{f(x)}{f(\frac{1}{2^{j+1}})}\{\chi_{[-\frac{1}{2^j},-\frac{1}{2^{j+1}}]}+\chi_{[\frac{1}{2^{j+1}},\frac{1}{2^{j}}]}\}\times 2^j,\ \ \lambda_j=\frac{f(\frac{1}{2^{j+1}})}{2^j}.\]
Thus, $f(x)=\sum_{j=1}^\infty\lambda_ja_j(x)$, and it is easy to verify that each $a_j$ is a $(1,\infty,0)$-atom. Notice that
\[\sum\limits_{j=1}^\infty|\lambda_j|=\sum\limits_{j=1}^\infty\frac{|f(\frac{1}{2^{j+1}})|}{2^j}\leq\sum\limits_{j=1}^\infty\frac{1}{2^j}
\cdot\frac{1}{\frac{1}{2^{j+1}}\log_2^{1+\varepsilon}2^{j+1}}=2\sum\limits_{j=1}^\infty\frac{1}{(j+1)^{1+\varepsilon}}<\infty,\] then we have $f\in H^1(\reals^n)$. Obviously, $f\notin L(\log L)(\reals^n)$.
\end{example}
Thus,  the $(H^1, L^{1,\infty})$ boundedness and the $L\log L$ type estimate of $T_b$ are independent in the sense that one can not cover the results of the other.\subsection{Commutators of multilinear operators}In recent years, the theory of multilinear Calder\'{o}n-Zygmund operators with standard kernels have been developed very quickly and a lot of works have been done. Among such achievements are the celebrated works of Coifman and Meyer \cite{RM1}, \cite{RM2} , \cite{RM3}, Christ
and Journ\'{e} \cite{Crist}, Kenig and Stein \cite{KS}, Grafakos and Torres \cite{GT1}, \cite{GT2}, and Lerner et al \cite{LOPT}. In order to state some known results, we need to introduce some definitions as follows:
\begin{definition}[\textbf{C-Z kernel of $\omega$ type}, \cite{LZ,MN}]
Let $\omega(t)$ be a non-negative and non-decreasing function on $\mathbb{R}^+$. Let $K(x, y_1, \cdot\cdot\cdot, y_m)$ be a locally integrable function defined away from the diagonal $x= y_1=\cdot\cdot\cdot=y_m$ in $(\reals^n)^{m+1}$. Denote $(x, \vec{y})=(x, y_1, \cdot\cdot\cdot, y_m)$, we say $K$ is an $m$-linear Calder\'{o}n-Zygmund kernel of $\omega$ type, if there exists a positive constants $C_0$ such that
\begin{equation}\label{k-sizes}
|K(x, \vec{y})|\leq\frac{C_0}{(\sum_{j=1}^m|x-y_j|)^{mn}},
\end{equation}
\begin{equation}\label{k-H1s}
|K(x, \vec{y})-K(x', \vec{y})|\leq\frac{C_0}{(\sum_{j=1}^m|x-y_j|)^{mn}}\omega\big(\frac{|x-x'|}{\sum_{j=1}^m|x-y_j|}\big),
\end{equation}
whenever $|x-x'|\leq\frac{1}{2}\max_{1\leq j\leq m}|x-y_j|$, and
\begin{align}\label{k-H2s}
|K(x, y_1, \cdot\cdot\cdot, y_i, \cdot\cdot\cdot, y_m)&-K(x, y_1, \cdot\cdot\cdot, y_i', \cdot\cdot\cdot, y_m)|\\ \notag
&\leq\frac{C_0}{(\sum_{j=1}^m|x-y_j|)^{mn}}\omega\big(\frac{|y_i-y_i'|}{\sum_{j=1}^m|x-y_j|}\big),
\end{align}
whenever $|y_i-y_i'|\leq\frac{1}{2}\max_{1\leq j\leq m}|x-y_j|$.
\end{definition}
\begin{definition}[\textbf{Multilinear C-Z singular integral operators}, \cite{LZ,MN}]\label{CZ}
Let $K(x, \vec{y})$ be a C-Z kernel of $\omega$ type. For any $\vec{f}=(f_1,\cdot\cdot\cdot,f_m)\in \mathscr{S}(\reals^n)\times\mathscr{S}(\reals^n)
\times\cdot\cdot\cdot\times\mathscr{S}(\reals^n)$ and all $x\notin\cap_{j=1}^m$ supp $f_j$, we define the multilinear Calder\'{o}n-Zygmund singular integral operators as follows:
\begin{equation*}\label{operator}
T(\vec{f})(x)=\int_{(\reals^n)^m}K(x, y_1, \cdot\cdot\cdot, y_m)f_1(y_1),\cdot\cdot\cdot, f_m(y_m)dy_1\cdot\cdot\cdot dy_m.
\end{equation*}
\end{definition}

\begin{definition}[\textbf{Commutators of Multilinear C-Z operators}] Let $b_j\in BMO(\reals^n)$ and $T$ be the operator defined in Definition \ref{CZ}. The commutators in the $j$-th entry and the iterated commutators of $T$ are defined by
 \begin{align}\label{jentry}
T_{\vec{b}}(\vec{f})(x)&=\sum_{i=1}^{m}T_{\vec{b}}^j(\vec{f})(x)\\ \notag &=\sum_{i=1}^{m}[b_j(x)T(f_1, \cdot\cdot\cdot, f_j, \cdot\cdot\cdot, f_m)(x)-T(f_1, \cdot\cdot\cdot, b_jf_j, \cdot\cdot\cdot, f_m)(x)]
\end{align}
and
 \begin{align}\label{iterated}
T_{\Pi b}(\vec{f})&=[b_1, [b_2, \cdot\cdot\cdot [b_{m-1},[b_m, T]_m, ]_{m-1}\cdot\cdot\cdot ]_2]_1(\vec{f})\\ \notag
&=\int_{(\reals^n)^m}\prod_{j=1}^{m}\big(b_j(x)-b_j(y_j)\big)K(x, y_1,\cdot\cdot\cdot,y_m)f_1(y_1)\cdot\cdot\cdot f_m(y_m)d\vec{y}.
\end{align}
\end{definition}
\begin{remark}Obviously, in the special case, $\omega(t)=t^\varepsilon$ for some $\varepsilon>0$, then the operator $T$ defined in Definition \ref{CZ} coincides with the standard multilinear Calder\'{o}n-Zygmund operator defined and studied by Grafakos and Torres \cite{GT1}. Moreover, if $\omega(t)=t^\varepsilon$, the weighted strong and $L(\log L)$ type endpoint estimates for $T_{\vec{b}}$ and $T_{\Pi b}$  have already been studied
in \cite{LOPT} and \cite{PPTT}, respectively. \end{remark}
 \begin{definition}[\textbf{Dini$(a)$ type conditions}]
Let $\omega (t)$ be a non-negative and non-decreasing function on $\mathbb{R}^+$. $\omega$ is said to satisfy the Dini$(a)$ condition if
$$\int_0^1\frac{\omega^a(t)}{t}dt<\infty.$$
$\omega$ is said to satisfy the $\log$-Dini$(a)$ condition if the following inequality holds:
\begin{equation}\label{condition1}\int_0^1\frac{\omega^a(t)}{t}\left(1+\log\frac{1}{t}\right)dt<\infty.
\end{equation}\end{definition}

\remark It's easy to see that the $\log$-Dini$(a)$ condition is more stronger than the Dini$(a)$ condition and if $0<a_1<a_2$, then Dini$(a_1)\subset$ Dini$(a_2)$.

In 2009, Maldonado and Naibo \cite{MN} showed that, when $\omega$ is concave and $\omega\in$ Dini$(1/2)$, the bilinear Calder\'{o}n-Zygmund operator of $\omega$ type is bounded from $L^1\times L^1$ to $L^{\frac{1}{2},\infty}$. In 2014, Lu and Zhang \cite{LZ} improved the results in \cite{MN} by removing the hypothesis that $\omega$ is concave and reduce the condition $\omega\in$ Dini$(1/2)$ to a weaker condition $\omega\in$ Dini$(1)$. Lu and Zhang \cite{LZ} also extended the weighted strong and $L(\log L)$ type endpoint estimates to the commutators defined in (\ref{jentry}) whenever $\omega$ satisfies the $\log$-Dini(1) condition, which is stronger than Dini$(1)$ condition but it is much weaker than the standard kernel $\omega(t)=t^\varepsilon$.
More previous works on the commutators of multilinear operators with $\omega(t)=t^\varepsilon$ can be found in \cite{PPTT}, \cite{PT}, \cite{SX}, \cite{T} and \cite{X}).
\subsection{Main results}This paper is concerned with the sharp endpoint estimates for both the commutator in the $j$-th entry defined in $(\ref{jentry})$ and iterated commutators defined in $(\ref{iterated})$ with a C-Z kernel of $\omega$ type. We show that they are bounded from product Hardy space $H^1\times\cdot\cdot\cdot\times H^1$ to weak $L^{\frac{1}{m},\infty}$ space, whenever the kernel satisfies a class of Dini type condition. However, the proof is very difficult and complicated. In particular, in the case of iterated commutators, sometimes, we need to control six summations and three integrals at the same time even for $m=2$. We formulate our main results as follows.
\begin{theorem} \label{thm1}
Let $T$ be a multilinear Calder\'{o}n-Zygmund operators with a C-Z kernel of $\omega$ type and $T_{\vec{b}}$ be the commutators of the $j$-th entries defined in $(\ref{jentry})$ with $\vec{b}\in BMO^m$. If $\omega(t)$ satisfies the $\log$-${\text{Dini}}(1)$ condition, then there exists a constant $C>0$, such that the following inequality holds
\begin{equation}\label{eq1}
|\{x\in\reals^n:|T_{\vec{b}}(\vec{f})(x)|>\lambda\}|\leq C_{\norm{\vec{b}}_{BMO^m}}\lambda^{-\frac{1}{m}}\prod_{j=1}^{m}\norm{f_j}_{H^{1}(\reals^n)}^{\frac{1}{m}}.
\end{equation}
 \end{theorem}
 With a more stronger condition assumed on the function $\omega(t)$ than in Theorem \ref{thm1}, but weaker condition than the standard kernel $\omega(t)=t^\varepsilon$, we obtain the following theorem for the iterated commutators.
 \begin{theorem} \label{thm1itrated}Let $\omega(t)$ be a doubling function, satisfying the $\log$-${\text{Dini}}(1/2m)$ condition, that is, $$\int_0^1\omega(t)^{\frac{1}{2m}}t^{-1}\left(1+\log\frac{1}{t}\right)dt<\infty.$$
Let $T$ be a multilinear Calder\'{o}n-Zygmund operators with a C-Z kernel of $\omega$ type and $T_{\Pi b}$ be the iterated commutators defined in $(\ref{iterated})$ with $\vec{b}\in BMO^m$. Then there exists a constant $C>0$, such that the following inequality holds
\begin{equation}\label{eq1t2}
|\{x\in\reals^n:|T_{\Pi b}(\vec{f})(x)|>\lambda\}|\leq  C_{\norm{\vec{b}}_{BMO^m}}\lambda^{-\frac{1}{m}}\prod_{j=1}^{m}\norm{f_j}_{H^{1}(\reals^n)}^{\frac{1}{m}}.
\end{equation}
 \end{theorem}
This article is organized as follows. In Section 2, the proof of Theorem \ref{thm1} will be given. Section 3 will be devoted to give the proof of Theorem \ref{thm1itrated}.
\section{Proofs of Theorem \ref{thm1}}
To prove Theorem \ref{thm1}, we need the following key lemma given by Chirst \cite{Ch}, which provides a foundation for our analysis.
\begin{lemma}\label{lemma}(\cite{Ch})
For any $\alpha>0$ and any finite collection of dyadic cubes $Q$ and associated positive scalars $\lambda_{Q}$, there exists a collection of pairwise disjoint dyadic cubes $S$ such that\par
$(1)$ $\sum \limits_{Q\subset S}\lambda_{Q}\leq 2^n\alpha|S|$, for all $S$;\par
$(2)$ $\sum |S|\leq\alpha^{-1}\sum \lambda_Q$;\par
$(3)$ $\norm{\sum \limits_{Q\nsubseteq \ \text{any} \ S}\lambda_{Q}|Q|^{-1}\chi_Q}_{L^{\infty}(\reals^n)}\leq\alpha$.
\end{lemma}
\begin{proof}[Proof of Thoerem~\ref{thm1}]

For simplicity, we only consider the case for $m=2$, because there is no essential difference for the general case.

Since~$T_{\vec{b}}$~is bounded from~$L^{2}(\reals^n)\times L^{2}(\reals^n)$~into~$L^1(\reals^n)$ \cite{LZ}, and finite sums of atoms are dense in~$H^1(\reals^n)$, we will work with such sums and will obtain desired estimates which is independent of the number of terms in each sum. Thus, for any given $f_j\in H^1(\reals^n)\  (j=1, 2)$, we may assume that $f_j=\sum_{k_j}\lambda_{k_j}a_{k_j}$ is a finite sum of $H^1$-atoms, where each $a_{k_j}$ is a $(1,\infty,0)$ atom, with $\sum_{k_j}|\lambda_{k_j}|\leq C\norm{f_j}_{H^1(\reals^n)}$. Set~$C_1=\norm{T_{\vec{b}}}_{L^2\times L^2\rightarrow L^{1,\infty}}$ and~$C_2=\norm{T}_{L^1\times L^1\rightarrow L^{\frac{1}{2},\infty}}$.
By linearity, it is sufficient to consider the commutator of $T$ with only one symbol, that is, for $\vec{b}=b\in BMO(\reals^n)$, we will consider the operator
$$T_b(f_1, f_2)(x)=b(x)T(f_1, f_2)(x)-T(bf_1, f_2)(x).$$
To prove inequality $(\ref{eq1})$, without loss of generality, we may assume that $\norm{f_j}_{H^1(\reals^n)}=1$ for $j=1, 2$. For fix $\lambda>0$, we only need to show that there is a constant $C>0$, independent on the variables and $f_{j} (j=1,2)$, such that
\begin{equation}\label{eq2}
|\{x\in\reals^n:|T_{b}(f_1, f_2)(x)|>\lambda\}|\leq C(C_0+C_1+C_2)^{1/2}\lambda^{-1/2}.
\end{equation}
Let $\gamma$ be a positive number to be determined later. For the finite collection of dyadic cubes $Q_{j,k_j}$, which associated with the positive scalars $\lambda_{Q_{j,k_j}}$ in the given atomic decomposition of $f_j$. Now, we take $\alpha=(\gamma\lambda)^{1/2}$ in lemma \ref{lemma}. Then, there exists a collection of pairwise disjoint dyadic cubes $S_{j,l_j}$, such that\par
\begin{align*}
&\text{(I)}\ \  \sum \limits_{Q_{j,k_j}\subset S_{j,l_j}}\lambda_{Q_{j,k_j}}\leq 2^n(\gamma\lambda)^{1/2}|S_{j,l_j}|, \text{\ \ for all\quad}S_{j,l_j};\\
&\text{(II)}\ \ \sum \limits_{S_{j,l_j}}|S_{j,l_j}|\leq(\gamma\lambda)^{-1/2}\sum \limits_{Q_{j,k_j}\subset S_{j,l_j}}\lambda_{Q_{j,k_j}};\\
&\text{(III)}\ \  \norm{\sum \limits_{Q_{j,k_j}\nsubseteq \ \text{any} \ S_{j,l_j}}\lambda_{Q_{j,k_j}}|Q_{j,k_j}|^{-1}\chi_{Q_{j,k_j}}}_{L^{\infty}(\reals^n)}\leq(\gamma\lambda)^{1/2}.
\end{align*}
Denote $S_{j,l_j}^* = 8\sqrt{n}S_{j,l_j}$, $S_j^*=\cup_{l_j}S_{j,l_j}^* $ for $j=1, 2$, and $S^*=\cup_{j=1}^2S_{j}^* $. Set $$h_j=\sum \limits_{S_{j,l_j}}\sum \limits_{Q_{j,k_j}\subset S_{j,l_j}}\lambda_{Q_{j,k_j}}a_{Q_{j,k_j}}\ \text{ and} \ \ g_j(x)=f_j(x)-h_j(x).$$
 By the definition of $g_j$ and $h_j$, $(III)$ and the properties of $(1,\infty,0)$ atoms, we have
 \begin{align*}
 &\norm{g_j}_{L^{\infty}(\reals^n)}\leq(\gamma\lambda)^{1/2};\quad \norm{g_j}_{L^{1}(\reals^n)}\leq\sum \limits_{Q_{j,k_j}\nsubseteq \text{any} S_{j,l_j}}|\lambda_{Q_{j,k_j}}|\leq \sum\limits_{k_j} |\lambda_{k_j}|\leq C\norm{f_j}_{H^1(\reals^n)};\\
 &\norm{h_j}_{L^{1}(\reals^n)}\leq\sum \limits_{S_{j,l_j}}\sum \limits_{Q_{j,k_j}\subset S_{j,l_j}}|\lambda_{Q_{j,k_j}}|\int_{\reals^n}|a_{Q_{j,k_j}}|dx\leq \sum\limits_{k_j} |\lambda_{k_j}|\leq C\norm{f_j}_{H^1(\reals^n)}.
\end{align*}
Now, we introduce some more notations as follows:
\begin{align*}
&E_1=\left\{x\in\reals^n: |T_b(g_1, g_2)(x)| >\lambda/4\right\};\quad E_2=\left\{x\in\reals^n\backslash S^*: |T_b(g_1, h_2)(x)| >\lambda/4\right\};\\
&E_3=\left\{x\in\reals^n\backslash S^*: |T_b(h_1, g_2)(x)| >\lambda/4\right\};\quad E_{4}=\left\{x\in\reals^n\backslash S^*: |T_b(h_1, h_2)(x)| >\lambda/4\right\}.
 \end{align*}
 By $(II)$, it follows that
 \begin{equation}\label{usedlatter1}
 |S^*|\leq\sum\limits_{j=1}^2|S_{j}^*|\leq\sum\limits_{j=1}^2\sum\limits_{S_{j,l_j}}|S_{j,l_j}^*|\leq C(\gamma\lambda)^{-1/2}\sum\limits_{j=1}^2\sum\limits_{Q_{j,l_j}\subset S_{j,l_j}}\lambda_{Q_{j,l_j}}\leq C(\gamma\lambda)^{-1/2}.
  \end{equation}
From the $L^2\times L^2\rightarrow L^{1,\infty}$ boundedness of $T_{\vec{b}}$, the Chebyshev inequality and $\norm{g_j}_{L^{\infty}(\reals^n)}\leq(\gamma\lambda)^{1/2}$, one may obtain
\begin{align}\label{usedlatter2}
|E_1|&\leq C_1\lambda^{-1}\norm{g_1}_{L^2(\reals^n)} \norm{g_2}_{L^2(\reals^n)}\leq C_1\lambda^{-1}(\gamma\lambda)^{\frac{1}{2}}\norm{g_1}_{L^1(\reals^n)}^{\frac{1}{2}} \norm{g_2}_{L^1(\reals^n)}^{\frac{1}{2}}\\ \notag &
\leq CC_1\gamma^{\frac{1}{2}}\lambda^{-1}\norm{f_1}_{H^1(\reals^n)}^{\frac{1}{2}} \norm{f_2}_{H^1(\reals^n)}^{\frac{1}{2}}= CC_1\gamma^{\frac{1}{2}}\lambda^{-\frac{1}{2}}
 \end{align}
Therefore, we get
 \begin{align}\label{final1}
|\{x\in\reals^n:|T_{b}(\vec{f})(x)|>\lambda\}|& \leq \sum\limits_{s=1}^{4}|E_s|+C|S^*|\\
 &\leq \sum\limits_{s=2}^{4}|E_s|+C(\gamma\lambda)^{-1/2}+CC_1\gamma^{\frac{1}{2}}\lambda^{-\frac{1}{2}}\notag.
  \end{align}
 Hence, to finish the proof of Theorem \ref{thm1}, we only need to consider the contributions of each $|E_s|$ for $2\le s\le 4$, separately.\par
 \noindent
 \bull
 \textbf {Estimate for $|E_2|$.} By the definition of $g_j$ and $h_j$, the moment condition of $H^1$-atoms, and employing the linearity of $T_b$, it now follows that
\begin{align}\label{eqx}
 &T_b(h_1,g_2)(x)\\ \notag
 &=\sum\limits_{S_{1,l_1}}\sum\limits_{Q_{1,k_1}\subset S_{1,l_1}}\lambda_{Q_{1,k_1}}\iint_{(\reals^n)^2}\big(b(x)-b_{Q_{1,k_1}}\big)\big(K(x, y_1, y_2)-K(x, c_{{1,k_1}}, y_2)\big)\\ \notag
 &\quad\times a_{Q_{1,k_1}}(y_1)g_2(y_2)d\vec{y}\\ \notag
 &\quad+\sum\limits_{S_{1,l_1}}\sum\limits_{Q_{1,k_1}\subset S_{1,l_1}}\lambda_{Q_{1,k_1}}\iint_{(\reals^n)^2}\big(b_{Q_{1,k_1}}-b(y_1)\big)K(x, y_1, y_2)a_{Q_{1,k_1}}(y_1)g_2(y_2)d\vec{y}\\ \notag
 &=: I_{2,1}(x)+I_{2,2}(x).
 \end{align}
Therefore, we have
\begin{align*}|E_2|&\leq |\{x\in\reals^n\backslash S^*: |I_{2,1}(x)| >\lambda/8\}|+|\{x\in\reals^n\backslash S^*: |I_{2,2}(x)| >\lambda/8\}|\\& :=|E_{2,1}|+|E_{2,2}|.\end{align*}
Thus, to show the contributions of $E_2$, it remains to discuss the contributions of $E_{2,1}$ and $E_{2,2}$, respectively.

To estimate $|E_{2,1}|$, we fix $k_1$ and denote $\mathscr{R}_{1, k_1}^i=(2^{i+2}\sqrt{n}Q_{1, k_1})\backslash (2^{i+1}\sqrt{n}Q_{1, k_1})$, $i=1,2,\cdot\cdot\cdot$. Then, it is obvious that $\reals^n\backslash S^*\subset\reals^n\backslash Q_{1, k_1}^*\subset \cup_{i=1}^{\infty}\mathscr{R}_{1, k_1}^i.$ Let $c_{1,k_1}$ be the center of cube $Q_{1,k_1}$, $l_{Q_{1,k_1}}$ be the side length of cube $Q_{1,k_1}$ Then, for any $y_1\in Q_{1, k_1}$ and $x\in\mathscr{R}_{1, k_1}^i$, we have
\begin{equation}\label{ineq-h}
|y_1-c_{1,k_1}|\leq \frac{1}{2}\sqrt{n} l_{Q_{1,k_1}}\ \ \text{and}\ \ |x-c_{1,k_1}|\geq 2^{i-1}\sqrt{n} l_{Q_{1,k_1}}.
\end{equation}
 By the Chebychev inequality and $(\ref{k-H2s})$, it follows that
 \begin{align}\label{eq3}
 |E_{2,1}| &\leq\frac{8 C_0}{\lambda}\norm{g_2}_{L^{\infty}}\sum \limits_{S_{1,l_1}}\sum \limits_{Q_{1,k_1}\subset S_{1,l_1}}|\lambda_{Q_{1,k_1}}|\int_{\reals^n\backslash S^*}\int_{\reals^n}\int_{\reals^n}|b(x)-b_{Q_{1, k_1}}|\\
 \notag&\quad\times\frac{|a_{1, k_1}(y_1)|}{(|x-y_1|+|x-y_2|)^{2n}}\omega\big(\frac{|y_1-c_{1,k_1}|}{|x-y_1|+|x-y_2|}\big)dy_1dy_2dx.
  \end{align}
  Since $\reals^n\backslash S^*\subset \cup_{i=1}^{\infty}\mathscr{R}_{1, k_1}^i$ and $\omega$ is nondecreasing, together with $(\ref{ineq-h})$ and notice that $a_{1,k_1}\in L^1(\reals^n)$, one obtains
  \begin{align*}
   &\int_{\reals^n\backslash S^*}\int_{\reals^n}\int_{\reals^n}|b(x)-b_{Q_{1, k_1}}|\frac{|a_{Q_{1, k_1}}(y_1)|}{(|x-y_1|+|x-y_2|)^{2n}}\omega\big(\frac{|y_1-c_{1,k_1}|}{|x-y_1|+|x-y_2|}\big)dy_1dy_2dx\\
   &\leq\sum\limits_{i=1}^{\infty}\int_{\mathscr{R}_{1, k_1}^i}\int_{\reals^n}\int_{\reals^n}|b(x)-b_{Q_{1, k_1}}|\frac{|a_{Q_{1, k_1}}(y_1)|}{(|x-y_1|+|x-y_2|)^{2n}}\omega\bigg(\frac{|y_1-c_{1,k_1}|}{|x-y_1|}\bigg)dy_1dy_2dx\\
   &\leq C\sum\limits_{i=1}^{\infty}\omega(2^{-i})\int_{\mathscr{R}_{1, k_1}^i}\int_{\reals^n}|b(x)-b_{Q_{1, k_1}}|\frac{|a_{Q_{1, k_1}}(y_1)|}{|x-y_1|^{n}}dy_1dx\\
   &\leq C\sum\limits_{i=1}^{\infty}\omega(2^{-i})\frac{1}{|2^{i+2}Q_{1,k_1}|}\int_{2^{i+2}Q_{1,k_1}}|b(x)-b_{Q_{1, k_1}}|dx\\
   &\leq C\sum\limits_{i=1}^{\infty}i\omega(2^{-i})\norm{\vec{b}}_*\leq C.
  \end{align*}
  Putting the above estimate into $(\ref{eq3})$ and noticing the fact that $\norm{g_j}_{L^{\infty}(\reals^n)}\leq(\gamma\lambda)^{1/2}$, we have
  \begin{align}\label{eq44}
 |E_{2,1}| &\leq \frac{CC_0}{\lambda}(\gamma\lambda)^{\frac{1}{2}}\sum \limits_{S_{1,l_1}}\sum \limits_{Q_{1,k_1}\subset S_{1,l_1}}|\lambda_{Q_{1,k_1}}|\leq CC_0\gamma^{\frac{1}{2}}\lambda^{-\frac{1}{2}}.
  \end{align}
Now, we are in the position to estimate $|E_{2,2}|$. The $L^1\times L^1\rightarrow L^{\frac{1}{2}, \infty}$ boundedness of $T$ implies that

  \begin{align}\label{eq5}
|E_{2,2}|   &\leq CC_2^{\frac{1}{2}}\lambda^{-\frac{1}{2}}\sum \limits_{S_{1,l_1}}\sum \limits_{Q_{1,k_1}\subset S_{1,l_1}}|\lambda_{Q_{1,k_1}}|\norm{\big(b(x)-b_{Q_{1, k_1}}\big)a_{Q_{1, k_1}}}_{L^1(\reals^n)}^{\frac{1}{2}}\norm{g_2}_{L^1(\reals^n)}^{\frac{1}{2}}\\ \notag
  &\leq CC_2^{\frac{1}{2}}\lambda^{-\frac{1}{2}}\sum \limits_{S_{1,l_1}}\sum \limits_{Q_{1,k_1}\subset S_{1,l_1}}|\lambda_{Q_{1,k_1}}|\big(\frac{1}{|Q_{1,k_1}|}\int_{Q_{1,k_1}}|b(y_1)-b_{Q_{1, k_1}}|dy_1\big)^{\frac{1}{2}}\norm{f_2}_{H^1(\reals^n)}^{\frac{1}{2}}\\ \notag
  &\leq CC_2^{\frac{1}{2}}\norm{\vec{b}}^{\frac{1}{2}}_{*}\lambda^{-\frac{1}{2}}\\ \notag
  &\leq CC_2^{\frac{1}{2}}\lambda^{-\frac{1}{2}}.
  \end{align}
  Therefore in all, Combining $(\ref{eq44})$ and the above estimate, we conclude that
  \begin{equation*}
  |E_2|\leq C(C_0\gamma^{\frac{1}{2}}\lambda^{-\frac{1}{2}}+C_2^{\frac{1}{2}}\lambda^{-\frac{1}{2}}).
  \end{equation*}
  \noindent
 \bull
 \textbf {Estimate for $|E_3|$.} The contributions of $E_3$ are treated in the same way as we deal with $|E_2|$. In fact,
   \begin{align*}
 &T_b(g_1,h_2)(x)\\
 &=\sum\limits_{S_{2,l_2}}\sum\limits_{Q_{2,k_2}\subset S_{2,l_2}}\lambda_{Q_{2,k_2}}\iint_{(\reals^n)^2}\big(b(x)-b_{Q_{2,k_2}}\big)\big(K(x, y_1, y_2)-K(x, y_1 , c_{{2,k_2}})\big)\\
 &\quad\times g_1(y_1)a_{Q_{2,k_2}}(y_2)d\vec{y}\\
 &\quad+\sum\limits_{S_{2,l_2}}\sum\limits_{Q_{2,k_2}\subset S_{2,l_2}}\lambda_{Q_{2,k_2}}\iint_{(\reals^n)^2}\big(b_{Q_{2,k_2}}-b(y_2)\big)K(x, y_1, y_2)g_1(y_1)a_{Q_{2,k_2}}(y_2)d\vec{y}\\
 &=: I_{3,1}(x)+I_{3,2}(x).
 \end{align*}
 Repeating the same steps as what we have done for $|E_2|$, we may obtain
 \begin{equation*}
  |E_3|\leq C(C_0\gamma^{\frac{1}{2}}\lambda^{-\frac{1}{2}}+C_2^{\frac{1}{2}}\lambda^{-\frac{1}{2}}).
  \end{equation*}
 \noindent
 \bull
 \textbf {Estimate for $|E_4|$.}
 First, we split $T_b(h_1,h_2)$ in the form as follows:
  \begin{align*}
 &T_b(h_1,h_2)(x)\\ \notag
 &=\sum\limits_{S_{1,l_1}}\sum\limits_{Q_{1,k_1}\subset S_{1,l_1}}\sum\limits_{S_{2,l_2}}\sum\limits_{Q_{2,k_2}\subset S_{2,l_2}}\iint_{(\reals^n)^2}\big(b(x)-b_{Q_{1,k_1}}\big)\big(K(x, y_1, y_2)-K(x, c_{{1,k_1}}, y_2)\big)\\
 &\quad\times \lambda_{Q_{1,k_1}}a_{Q_{1,k_1}}(y_1)\lambda_{Q_{2,k_2}}a_{Q_{2,k_2}}(y_2)d\vec{y}\\
 &\quad+\sum\limits_{S_{1,l_1}}\sum\limits_{Q_{1,k_1}\subset S_{1,l_1}}\iint_{(\reals^n)^2}\big(b_{Q_{1,k_1}}-b(y_1)\big)K(x, y_1, y_2)\lambda_{Q_{1,k_1}}a_{Q_{1,k_1}}(y_1)h_2(y_2)d\vec{y}\\ \notag
 &=: I_{4,1}(x)+I_{4,2}(x).
 \end{align*}
 Hence, we have
\begin{equation}\label{eqy}
|E_4|\leq |\{x\in\reals^n\backslash S^*: |I_{4,1}(x)| >\lambda/8\}|+|\{x\in\reals^n\backslash S^*: |I_{4,2}(x)| >\lambda/8\}|.\end{equation}
For fixed $k_2$, denote $\mathscr{R}_{2, k_2}^h=(2^{h+2}\sqrt{n}Q_{2, k_2})\backslash (2^{h+1}\sqrt{n}Q_{2, k_2})$, $h=1,2,\cdot\cdot\cdot$. Recall the definition of $\mathscr{R}_{1, k_1}^i$, it is easy to check
\begin{equation*}
(S^*)^c:=\reals^n\backslash S^*\subset\reals^n\backslash(Q^*_{1,k_1}\bigcup Q^*_{1,k_2})\subset\bigcup_{h=1}^{\infty}\bigcup_{i=1}^{\infty}\big(\mathscr{R}_{1, k_1}^i\bigcap\mathscr{R}_{2, k_2}^h\big).
\end{equation*}
Therefore, one may obtain that
\begin{equation}\label{key}
(S^*)^c=(S^*)^c\bigcap\bigg(\bigcup_{h=1}^{\infty}\bigcup_{i=1}^{\infty}\big(\mathscr{R}_{1, k_1}^i\bigcap\mathscr{R}_{2, k_2}^h\big)\bigg)
=\bigcup_{h=1}^{\infty}\bigcup_{i=1}^{\infty}\bigg((S^*)^c\bigcap\big(\mathscr{R}_{1, k_1}^i\bigcap\mathscr{R}_{2, k_2}^h\big)\bigg).
\end{equation}

 By the Chebychev inequality, $(\ref{k-H2s})$ and $(\ref{key})$, it follows that
 \begin{align}\label{eq6}
 &|\{x\in\reals^n\backslash S^*: |I_{4,1}(x)| >\lambda/8\}|\\ \notag
 \notag&\leq\frac{8 C_0}{\lambda}\sum \limits_{S_{1,l_1}}\sum \limits_{Q_{1,k_1}\subset S_{1,l_1}}\sum \limits_{S_{2,l_2}}\sum \limits_{Q_{2,k_2}\subset S_{2,l_2}} \int_{\reals^n\backslash S^*}\iint_{(\reals^n)^2}|b(x)-b_{Q_{1, k_1}}|\\ \notag
 &\quad\times\frac{|\lambda_{Q_{1,k_1}}||a_{Q_{1, k_1}}(y_1)||\lambda_{Q_{2,k_2}}||a_{Q_{2, k_2}}(y_2)|}{(|x-y_1|+|x-y_2|)^{2n}}\omega\big(\frac{|y_1-c_{1,k_1}|}{|x-y_1|+|x-y_2|}\big)dy_1dy_2dx.
 \end{align}
 Moreover, by $(\ref{key})$, the integrals in the above summations can be controlled by:\begin{align}\label{eqqqq}
 &\sum\limits_{i=1}^{\infty}\sum\limits_{h=1}^{\infty}\int_{(S^*)^c\cap\mathscr{R}_{1, k_1}^i\cap\mathscr{R}_{2, k_2}^h} \iint_{(\reals^n)^2}|b(x)-b_{Q_{1, k_1}}|\\ \notag
 \notag&\quad\times\frac{|\lambda_{Q_{1,k_1}}||a_{Q_{1, k_1}}(y_1)||\lambda_{Q_{2,k_2}}||a_{Q_{2, k_2}}(y_2)|}{(|x-y_1|+|x-y_2|)^{2n}}\omega\big(\frac{|y_1-c_{1,k_1}|}{|x-y_1|}\big)dy_1dy_2dx\\ \notag
 \notag&\leq\sum\limits_{i=1}^{\infty}\sum\limits_{h=1}^{\infty}\omega(2^{-i})\int_{(S^*)^c\cap\mathscr{R}_{1, k_1}^i\cap\mathscr{R}_{2, k_2}^h} \iint_{(\reals^n)^2}|b(x)-b_{Q_{1, k_1}}|\\ \notag
 \notag&\quad\times|\lambda_{Q_{1,k_1}}||a_{Q_{1, k_1}}(y_1)||\lambda_{Q_{2,k_2}}||a_{Q_{2, k_2}}(y_2)|\sup\limits_{y_1, y_2\in S}\frac{1}{(|x-y_1|+|x-y_2|)^{2n}}dy_1dy_2dx.
  \end{align}
 For fixed $x\in (S^*)^c$, and any $y_1, y_2\in S$, observe that \[\inf\limits_{y_1\in S}|x-y_1|\approx|x-y_1|,\qquad \inf\limits_{y_2\in S}|x-y_2|\approx|x-y_2|.\]
  This implies that
 \begin{align} \label{rel1}
 \sup\limits_{y_1, y_2\in S}\frac{1}{(|x-y_1|+|x-y_2|)^{2n}}&=\frac{1}{(\inf\limits_{y_1\in S}|x-y_1|+\inf\limits_{y_2\in S}|x-y_2|)^{2n}}\\ \notag
 &\approx\frac{1}{(|x-y_1|+|x-y_2|)^{2n}}.
 \end{align}
 Note that $\{S_{j,l_j}\}_{l_j}$ are pairwise disjoint dyadic cubes, by $(I)$ and $(\ref{rel1})$, it now follows that
 \begin{align} \label{eq8}
 &\sum \limits_{S_{2,l_2}}\sum \limits_{Q_{2,k_2}\subset S_{2,l_2}}\int_{\reals^n}|\lambda_{Q_{2,k_2}}||a_{Q_{2, k_2}}(y_2)| \sup\limits_{y_1, y_2\in S}\frac{1}{(|x-y_1|+|x-y_2|)^{2n}}dy_2\\ \notag
 &=\sum \limits_{S_{2,l_2}}\sum \limits_{Q_{2,k_2}\subset S_{2,l_2}}|\lambda_{Q_{2,k_2}}|\sup\limits_{y_1, y_2\in S}\frac{1}{(|x-y_1|+|x-y_2|)^{2n}}\int_{\reals^n}|a_{Q_{2, k_2}}(y_2)|dy_2\\ \notag
 &\leq C\sum \limits_{S_{2,l_2}}\bigg(\sum \limits_{Q_{2,k_2}\subset S_{2,l_2}}|\lambda_{Q_{2,k_2}}|\bigg)\sup\limits_{y_1, y_2\in S}\frac{1}{(|x-y_1|+|x-y_2|)^{2n}}\\ \notag
 &\leq \sum \limits_{S_{2,l_2}}2^n(\gamma\lambda)^{1/2}|S_{2,l_2}|\sup\limits_{y_1, y_2\in S}\frac{1}{(|x-y_1|+|x-y_2|)^{2n}}\\ \notag
 &\leq C(\gamma\lambda)^{1/2}\sum \limits_{S_{2,l_2}}\int_{S_{2,l_2}}\frac{1}{(|x-y_1|+|x-y_2|)^{2n}}dy_2\\ \notag
 &\leq C(\gamma\lambda)^{1/2}\frac{1}{|x-y_1|^n}.
 \end{align}

 Combing $(\ref{eq6})$, $(\ref{eqqqq})$ and $(\ref{eq8})$, we obtain
 \begin{align}\label{eq9}
 &|\{x\in\reals^n\backslash S^*: |I_{4,1}(x)| >\lambda/8\}|\\ \notag
 &\leq CC_0\gamma^{\frac{1}{2}}\lambda^{-\frac{1}{2}}\sum \limits_{S_{1,l_1}}\sum \limits_{Q_{1,k_1}\subset S_{1,l_1}}\sum\limits_{i=1}^{\infty}\sum\limits_{h=1}^{\infty}\omega(2^{-i})\int_{(S^*)^c\cap\mathscr{R}_{1, k_1}^i\cap\mathscr{R}_{2, k_2}^h}\int_{\reals^n}|b(x)-b_{Q_{1, k_1}}| \\ \notag
 \notag&\quad\times\frac{|\lambda_{Q_{1,k_1}}||a_{Q_{1, k_1}}(y_1)|}{|x-y_1|^{2n}}dy_1dx\\ \notag
 &\leq CC_0\gamma^{\frac{1}{2}}\lambda^{-\frac{1}{2}}\sum \limits_{S_{1,l_1}}\sum \limits_{Q_{1,k_1}\subset S_{1,l_1}}\sum\limits_{i=1}^{\infty}\omega(2^{-i})\int_{\mathscr{R}_{1, k_1}^i}\int_{\reals^n}|b(x)-b_{Q_{1, k_1}}|\frac{|\lambda_{Q_{1,k_1}}||a_{Q_{1, k_1}}(y_1)| }{|x-y_1|^{2n}}dy_1dx\\ \notag
 &\leq CC_0\gamma^{\frac{1}{2}}\lambda^{-\frac{1}{2}}\sum \limits_{S_{1,l_1}}\sum \limits_{Q_{1,k_1}\subset S_{1,l_1}}|\lambda_{Q_{1,k_1}}|\sum\limits_{i=1}^{\infty}\omega(2^{-i})\frac{1}{|2^{i+2}Q_{1,k_1}|}\int_{2^{i+2}Q_{1,k_1}}|b(x)-b_{Q_{1, k_1}}|dx\\ \notag
  &\leq CC_0\norm{\vec{b}}_*\gamma^{\frac{1}{2}}\lambda^{-\frac{1}{2}}\sum \limits_{S_{1,l_1}}\sum \limits_{Q_{1,k_1}\subset S_{1,l_1}}|\lambda_{Q_{1,k_1}}|\sum\limits_{i=1}^{\infty}\omega(2^{-i})i\\ \notag
  &\leq CC_0\gamma^{\frac{1}{2}}\lambda^{-\frac{1}{2}}.
 \end{align}
 The estimate of $|\{x\in\reals^n\backslash S^*: |I_{4,2}(x)| >\lambda/8\}|$ is similar to $(\ref{eq5})$. In fact, we only need to replace $g_2$ by $h_2$ in $(\ref{eq5})$, and note that $\norm{h_2}_{L^1}\leq C\norm{f_2}_{H^1}$, we have
\begin{equation}\label{eq10}
|\{x\in\reals^n\backslash S^*: |I_{4,2}(x)| >\lambda/8\}|\leq CC_2^{\frac{1}{2}}\lambda^{-\frac{1}{2}}.
\end{equation}
Putting $(\ref{eq9})$ and $(\ref{eq10})$ into $(\ref{eqy})$, it yields that
\begin{equation*}
  |E_4|\leq C(C_0\gamma^{\frac{1}{2}}\lambda^{-\frac{1}{2}}+C_2^{\frac{1}{2}}\lambda^{-\frac{1}{2}}).
  \end{equation*}
   Thus, we have proved that
  \begin{equation}\label{final2}
  |E_s|\leq C(C_0\gamma^{\frac{1}{2}}\lambda^{-\frac{1}{2}}+C_2^{\frac{1}{2}}\lambda^{-\frac{1}{2}})\ \ \text{for}\ \ s=2, 3, 4.
  \end{equation}
 Set $\gamma=(C_0+C_1+C_2)^{-1}$, by $(\ref{final1})$ and $(\ref{final2})$, we have
  \begin{align*}
 |\{x\in\reals^n:|T_{b}(\vec{f})(x)|>\lambda\}|&\leq \sum\limits_{s=2}^{4}|E_s|+C(\gamma\lambda)^{-1/2}+CC_1\gamma^{\frac{1}{2}}\lambda^{-\frac{1}{2}}\\ \notag &\leq C(C_0+C_1+C_2)^{1/2}\lambda^{-1/2}.
  \end{align*}
The proof of (\ref{eq2}) is finished. Since we have reduced the proof of Theorem $\ref{thm1}$ to (\ref{eq2}), the proof of Theorem $\ref{thm1}$ is complete.
\end{proof}
\section{Proof of Thoerem~\ref{thm1itrated}}
\begin{proof}[Proof of Thoerem~\ref{thm1itrated}]
We will also only consider Theorem \ref{thm1itrated} for the case $m = 2$. Thus, it is sufficient to consider the following operator:
\begin{align*}
T_{\pi b}(f_1, f_2)(x)&=[b_1,[b_2, T]_2,]_1(f_1,f_2)\\ \notag
&=\int_{(\reals^n)^m}\prod_{j=1}^{2}\big(b_j(x)-b_j(y_j)\big)K(x, y_1,y_2)f_1(y_1) f_2(y_2)dy_1dy_2,
\end{align*}
where $f_j\in H^1(\reals^n) \ (j=1, 2)$ with $\norm{f_j}_{H^1(\reals^n)}=1$ for $j=1, 2$.
Since~$T_{\pi b}(f_1,f_2)(x)$~is bounded from~$L^{2}(\reals^n)\times L^{2}(\reals^n)$~into~$L^1(\reals^n)$ (see \cite {PPTT} for the case of the standard kernel $\omega(t)=t^\varepsilon$, and the C-Z kernel of $\omega$ type in \cite{yan}), we may set~$C_1'=\norm{T_{\pi b}}_{L^2\times L^2\rightarrow L^{1,\infty}}$. Recall~$C_2=\norm{T}_{L^1\times L^1\rightarrow L^{\frac{1}{2},\infty}}$, following the same argument as in the proof of Theorem \ref{thm1}, it is also sufficient to show that
\begin{equation}\label{eq2i}
|\{x\in\reals^n:|T_{\pi b}(f_1, f_2)(x)|>\lambda\}|\leq C(C_0+C_1'+C_2)^{1/2}\lambda^{-1/2}.
\end{equation}
The same decomposition for $f_j\in H^1(\reals^n) \ (j=1, 2)$ as in Theorem \ref{thm1} yields that
\begin{align}\label{decom}
&h_j=\sum \limits_{S_{j,l_j}}\sum \limits_{Q_{j,k_j}\subset S_{j,l_j}}\lambda_{Q_{j,k_j}}a_{Q_{j,k_j}},\ \ \ \ f_j(x)=g_j(x)+h_j(x),
\end{align}
where $g_j$ and $h_j$ enjoy the same properties as in Theorem \ref{thm1}.\par
 With abuse of notations, we may still set
\begin{align*}
&E_1=\left\{x\in\reals^n: |T_{\pi b}(g_1, g_2)(x)| >\lambda/4\right\};\\
&E_2=\left\{x\in\reals^n\backslash S^*: |T_{\pi b}(g_1, h_2)(x)| >\lambda/4\right\};\\
&E_3=\left\{x\in\reals^n\backslash S^*: |T_{\pi b}(h_1, g_2)(x)| >\lambda/4\right\};\\
&E_{4}=\left\{x\in\reals^n\backslash S^*: |T_{\pi b}(h_1, h_2)(x)| >\lambda/4\right\}.
 \end{align*}
Then, $(\ref{usedlatter1})$ still implies that \[|S^*|\leq C(\gamma\lambda)^{-1/2}.\]
  Note that ~$C_1'=\norm{T_{\pi b}}_{L^2\times L^2\rightarrow L^{1,\infty}}$, repeating the arguments as in the estimates of $(\ref{usedlatter2})$, we may obtain
  \[|E_1|\leq CC_1'\gamma^{\frac{1}{2}}\lambda^{-\frac{1}{2}}.\]
  Therefore,
 \begin{align*}
 |\{x\in\reals^n:|T_{\pi b}(\vec{f})(x)|>\lambda\}|&\leq \sum\limits_{s=1}^{4}|E_s|+C|S^*|\\
 &\leq \sum\limits_{s=2}^{4}|E_s|+C(\gamma\lambda)^{-1/2}+CC_1\gamma^{\frac{1}{2}}\lambda^{-\frac{1}{2}}.
  \end{align*}
 Thus, to show Theorem \ref{thm1itrated} is true, we only have to show that
  \begin{equation}\label{lasti2i}
  |E_s|\leq C(C_0\gamma^{\frac{1}{2}}\lambda^{-\frac{1}{2}}+C_2^{-\frac{1}{2}}\lambda^{-\frac{1}{2}}),\ \ \text{for}\ \  s=2, 3, 4.
  \end{equation}
 In fact, let $\gamma=(C_0+C_1'+C_2)^{-\frac{1}{2}}$, it's easy to check that the inequality $(\ref{eq2i})$ is true.\\
 \noindent
 \bull
 \textbf {Estimate for $|E_2|$.} Employing the linearity of $T_{\pi b}$ and the atomic decomposition of $h_1$, we may decompose $T_{\pi b}(h_1, g_2)$ by:
\begin{align*}
&T_{\pi b}(h_1, g_2)(x)\\&=\int_{(\reals^n)^m}\prod_{j=1}^{2}\big(b_j(x)-b_j(y_j)\big)K(x, y_1,y_2)h_1(y_1) g_2(y_2)dy_1dy_2\\
&=\sum \limits_{S_{1,l_1}}\sum \limits_{Q_{1,k_1}\subset S_{1,l_1}}\lambda_{Q_{1,k_1}}\big(b_1(x)b_2(x)T(a_{Q_{1,k_1}},g_2)(x)-b_2(x)T(b_1a_{Q_{1,k_1}},g_2)(x)\\
&\quad-b_1(x)T(a_{Q_{1,k_1}}, b_2g_2)(x)+T(b_1a_{Q_{1,k_1}}, b_2g_2)(x)\big)\\
&=\sum \limits_{S_{1,l_1}}\sum \limits_{Q_{1,k_1}\subset S_{1,l_1}}\lambda_{Q_{1,k_1}}\big(b_1(x)-b_{1,Q_{1,k_1}}\big)\big(b_2(x)-b_{2,Q_{1,k_1}}\big)T(a_{Q_{1,k_1}},g_2)(x)\\
&\quad-\sum \limits_{S_{1,l_1}}\sum \limits_{Q_{1,k_1}\subset S_{1,l_1}}\lambda_{Q_{1,k_1}}\big(b_2(x)-b_{2,Q_{1,k_1}}\big)T((b_1-b_{1,Q_{1,k_1}})a_{Q_{1,k_1}},g_2)(x)\\
&\quad-\sum \limits_{S_{1,l_1}}\sum \limits_{Q_{1,k_1}\subset S_{1,l_1}}\lambda_{Q_{1,k_1}}\big(b_1(x)-b_{1,Q_{1,k_1}}\big)T(a_{Q_{1,k_1}}, (b_2-b_{2,Q_{1,k_1}})g_2)(x)\\
&\quad+\sum\limits_{S_{1,l_1}}\sum \limits_{Q_{1,k_1}\subset S_{1,l_1}}\lambda_{Q_{1,k_1}}T((b_1-b_{1,Q_{1,k_1}})a_{Q_{1,k_1}}, (b_2-b_{2,Q_{1,k_1}})g_2)(x)\\
&=:I_{2,1}(x)+I_{2,2}(x)+I_{2,3}(x)+I_{2,4}(x).
\end{align*}
Thus, the contributions of $E_2$ can be divide into four parts.
\begin{align*}
 |E_2|&=|\{x\in\reals^n:|T_{\pi b}(g_1,h_2)(x)|>\lambda/4\}|\\
 &\leq |\{x\in\reals^n:|I_{2,1}(x)|>\lambda/16\}|+|\{x\in\reals^n:|I_{2,2}(x)|>\lambda/16\}|\\
 &\quad+|\{x\in\reals^n:|I_{2,3}(x)|>\lambda/16\}|+|\{x\in\reals^n:|I_{2,4}(x)|>\lambda/16\}|\\
 &=:|E_{2,1}|+|E_{2,2}|+|E_{2,3}|+|E_{2,4}|.
\end{align*}
 By the definition of $I_{2,1}$ and the moment condition of $H^1$-atoms, we have
 \begin{align*}
I_{2,1}(x)&=\sum \limits_{S_{1,l_1}}\sum \limits_{Q_{1,k_1}\subset S_{1,l_1}}\lambda_{Q_{1,k_1}}\big(b_1(x)-b_{1,Q_{1,k_1}}\big)\big(b_2(x)-b_{2,Q_{1,k_1}}\big)\\
&\quad\times\iint_{(\reals^n)^2}\big(K(x,y_1,y_2)-K(x,c_{1,k_1},y_2)\big)a_{Q_{1,k_1}}(y_1)g_2(y_2)dy_1dy_2.
  \end{align*}
  Putting the above identity into the definition of $|E_{2,1}|$ and note that $\norm{g_2}_{L^{\infty}(\reals^n)}\leq(\gamma\lambda)^{1/2}$, $\reals^n\backslash S^*\subset \cup_{i=1}^{\infty}\mathscr{R}_{1, k_1}^i$, together with the Chebyshev inequality and condition $(\ref{k-H2s})$, we have
 \begin{align}\label{i213}
 |E_{2,1}|&\leq\frac{16}{\lambda}\sum \limits_{S_{1,l_1}}\sum \limits_{Q_{1,k_1}\subset S_{1,l_1}}|\lambda_{Q_{1,k_1}}|\int_{(S^*)^c}\iint_{(\reals^n)^2}|b_1(x)-b_{1,Q_{1,k_1}}||b_2(x)-b_{2,Q_{1,k_1}}|\\ \notag
&\quad\times|K(x,y_1,y_2)-K(x,c_{1,k_1},y_2)||a_{Q_{1,k_1}}(y_1)||g_2(y_2)|dy_1dy_2dx\\ \notag
&\leq CC_0\lambda^{1/2}\gamma^{-1/2}\sum \limits_{S_{1,l_1}}\sum \limits_{Q_{1,k_1}\subset S_{1,l_1}}|\lambda_{Q_{1,k_1}}|\sum\limits_{i=1}^{\infty}\int_{\mathscr{R}_{1, k_1}^i}\iint_{(\reals^n)^2}|b_1(x)-b_{1,Q_{1,k_1}}|\\ \notag
&\quad\times|b_2(x)-b_{2,Q_{1,k_1}}|\frac{|a_{1, k_1}(y_1)|}{(|x-y_1|+|x-y_2|)^{2n}}\omega\big(\frac{|y_1-c_{1,k_1}|}{|x-y_1|+|x-y_2|}\big)dy_1dy_2dx.
  \end{align}
  By $(\ref{ineq-h})$ and the non-decreasing property of $\omega$, we have
\begin{align}\label{i21}
   |E_{2,1}|&\leq CC_0\gamma^{1/2}\lambda^{-1/2}\sum \limits_{S_{1,l_1}}\sum \limits_{Q_{1,k_1}\subset S_{1,l_1}}|\lambda_{Q_{1,k_1}}|\sum\limits_{i=1}^{\infty}\int_{\mathscr{R}_{1, k_1}^i}\iint_{(\reals^n)^2}|b_1(x)-b_{1,Q_{1,k_1}}|\\ \notag
&\quad\times|(b_2(x)-b_{2,Q_{1,k_1}}|\frac{|a_{1, k_1}(y_1)|}{(|x-y_1|+|x-y_2|)^{2n}}\omega(2^{-i})dy_1dy_2dx\\ \notag
&\leq CC_0\gamma^{1/2}\lambda^{-1/2}\sum \limits_{S_{1,l_1}}\sum \limits_{Q_{1,k_1}\subset S_{1,l_1}}|\lambda_{Q_{1,k_1}}|\sum\limits_{i=1}^{\infty}\int_{(S^*)^c\cap\mathscr{R}_{1, k_1}^i}\iint_{(\reals^n)^2}|b_1(x)-b_{1,Q_{1,k_1}}|\\ \notag
&\quad\times|b_2(x)-b_{2,Q_{1,k_1}}|\frac{|a_{1, k_1}(y_1)|}{|x-y_1|^{n}}\omega(2^{-i})dy_1dx\\ \notag
&\leq CC_0\gamma^{1/2}\lambda^{-1/2}\sum \limits_{S_{1,l_1}}\sum \limits_{Q_{1,k_1}\subset S_{1,l_1}}|\lambda_{Q_{1,k_1}}|\sum\limits_{i=1}^{\infty}\int_{(S^*)^c\cap\mathscr{R}_{1, k_1}^i}\int_{\reals^n}|b_1(x)-b_{1,Q_{1,k_1}}|\\ \notag
&\quad\times|b_2(x)-b_{2,Q_{1,k_1}}|\frac{|a_{1, k_1}(y_1)|}{|2^{i+2}\sqrt{n}Q_{1,k_1}|}\omega(2^{-i})dy_1dx\\ \notag &
\leq CC_0\gamma^{1/2}\lambda^{-1/2}\sum \limits_{S_{1,l_1}}\sum \limits_{Q_{1,k_1}\subset S_{1,l_1}}|\lambda_{Q_{1,k_1}}|\sum\limits_{i=1}^{\infty}\omega(2^{-i})\frac{1}{|2^{i+2}\sqrt{n}Q_{1,k_1}|}\end{align}
\begin{align*}
&\quad\times\int_{\mathscr{R}_{1, k_1}^i}|b_1(x)-b_{1,Q_{1,k_1}}||b_2(x)-b_{2,Q_{1,k_1}}|dx.
\end{align*}
By the H\"{o}lder inequality, one obtains
 \begin{align}\label{i212}
 &\frac{1}{|2^{i+2}\sqrt{n}Q_{1,k_1}|}\int_{\mathscr{R}_{1,k_1}^i}|b_1(x)-b_{1,Q_{1,k_1}}||b_2(x)-b_{2,Q_{1,k_1}}|dx\\ \notag
 &\leq\bigg(\frac{1}{|2^{i+2}\sqrt{n}Q_{1,k_1}|}\int_{2^{i+2}\sqrt{n}Q_{1,k_1}}|b_1(x)-b_{1,Q_{1,k_1}}|^2 dx\bigg)^{1/2}\\ \notag
 &\quad\times\bigg(\frac{1}{|2^{i+2}\sqrt{n}Q_{1,k_1}|}\int_{2^{i+2}\sqrt{n}Q_{1,k_1}}|b_2(x)-b_{2,Q_{1,k_1}}|^2 dx\bigg)^{1/2}\\ \notag
 &\leq Ci\norm{b}_*.
 \end{align}
 Combing $(\ref{i21})$ and $(\ref{i212})$, we get
 \[|E_{2,1}|\leq CC_0\gamma^{1/2}\lambda^{-1/2}\sum \limits_{S_{1,l_1}}\sum \limits_{Q_{1,k_1}\subset S_{1,l_1}}|\lambda_{Q_{1,k_1}}|\sum\limits_{i=1}^{\infty}\omega(2^{-i})i\leq CC_0\gamma^{1/2}\lambda^{-1/2}.\]
 \qquad Now we begin to estimate $|E_{2,2}|$.\par
\quad Similarly as we deal $|E_{2,1}|$, and together with the size condition of $H^1$-atoms, it follows that
  \begin{align*}
  |E_{2,2}|&\leq CC_0\gamma^{1/2}\lambda^{-1/2}\sum \limits_{S_{1,l_1}}\sum \limits_{Q_{1,k_1}\subset S_{1,l_1}}|\lambda_{Q_{1,k_1}}|\sum\limits_{i=1}^{\infty}\int_{(S^*)^c\cap\mathscr{R}_{1, k_1}^i}\iint_{(\reals^n)^2}|b_1(y_1)-b_{1,Q_{1,k_1}}|\\ \notag
&\quad\times|b_2(x)-b_{2,Q_{1,k_1}}|\frac{|a_{Q_{1, k_1}}(y_1)|}{(|x-y_1|+|x-y_2|)^{2n}}\omega\big(\frac{|y_1-c_{1,k_1}|}{|x-y_1|+|x-y_2|}\big)dy_1dy_2dx\\
&\leq CC_0\gamma^{1/2}\lambda^{-1/2}\sum \limits_{S_{1,l_1}}\sum \limits_{Q_{1,k_1}\subset S_{1,l_1}}|\lambda_{Q_{1,k_1}}|\sum\limits_{i=1}^{\infty}\int_{(S^*)^c\cap\mathscr{R}_{1, k_1}^i}\int_{\reals^n}|b_1(y_1)-b_{1,Q_{1,k_1}}|\\ \notag
&\quad\times|b_2(x)-b_{2,Q_{1,k_1}}|\frac{1}{(|x-y_1|)^{n}|Q_{1, k_1}|}\omega(2^{-i})dy_1dx\\
&\leq CC_0\gamma^{1/2}\lambda^{-1/2}\norm{b_1}_*\sum \limits_{S_{1,l_1}}\sum \limits_{Q_{1,k_1}\subset S_{1,l_1}}|\lambda_{Q_{1,k_1}}|\sum\limits_{i=1}^{\infty}\omega(2^{-i})\frac{1}{(|2^{i+2}Q_{1,k_1}|)^{n}}\\ \notag
&\quad\times\int_{(S^*)^c\cap\mathscr{R}_{1, k_1}^i}|b_2(x)-b_{2,Q_{1,k_1}}|dx\\
&\leq CC_0\gamma^{1/2}\lambda^{-1/2}\norm{b_1}_*\norm{b_2}_*\sum \limits_{S_{1,l_1}}\sum \limits_{Q_{1,k_1}\subset S_{1,l_1}}|\lambda_{Q_{1,k_1}}|\sum\limits_{i=1}^{\infty}\omega(2^{-i})i\\
&\leq CC_0\gamma^{1/2}\lambda^{-1/2}.
\end{align*}
The estimate for $|E_{2,3}|$ is more complicated, and we need to split the domain of variable $y_2$. First, similarly as we deal with $|E_{2,1}|$ in (\ref{i213}) and (\ref{i21}), we may get
\begin{align*}
|E_{2,3}|&\leq CC_0\gamma^{1/2}\lambda^{-1/2}\sum \limits_{S_{1,l_1}}\sum \limits_{Q_{1,k_1}\subset S_{1,l_1}}|\lambda_{Q_{1,k_1}}|\sum\limits_{i=1}^{\infty}\int_{(S^*)^c\cap\mathscr{R}_{1, k_1}^i}\iint_{(\reals^n)^2}|b_1(x)-b_{1,Q_{1,k_1}}|\\ \notag
&\quad\times|b_2(y_2)-b_{2,Q_{1,k_1}}|\frac{|a_{Q_{1, k_1}}(y_1)|}{(|x-y_1|+|x-y_2|)^{2n}}\omega\big(\frac{|y_1-c_{1,k_1}|}{|x-y_1|+|x-y_2|}\big)dy_1dy_2dx.
\end{align*}
  Denote $\mathscr{R}_{1, k_1}^h=(2^{h+2}\sqrt{n}Q_{1, k_1})\backslash (2^{h+1}\sqrt{n}Q_{1, k_1})$ and recall that $Q_{1, k_1}^*=4\sqrt{n}Q_{1, k_1}$, then \[y_2\in\reals^n\subset\big(\cup_{h=1}^{\infty}\mathscr{R}_{1, k_1}^h\big)\cup Q_{1, k_1}^*.\]
  Thus $|E_{2,3}|$ can be controlled by
   \begin{align*}
 &CC_0\gamma^{1/2}\lambda^{-1/2}\sum \limits_{S_{1,l_1}}\sum \limits_{Q_{1,k_1}\subset S_{1,l_1}}|\lambda_{Q_{1,k_1}}|\sum\limits_{i=1}^{\infty}\int_{(S^*)^c\cap\mathscr{R}_{1, k_1}^i}\int_{\cup_{i=1}^{\infty}\mathscr{R}_{1, k_1}^h}\int_{\reals^n}|b_1(x)-b_{1,Q_{1,k_1}}|\\ \notag
&\quad\times|b_2(y_2)-b_{2,Q_{1,k_1}}|\frac{|a_{Q_{1, k_1}}(y_1)|}{(|x-y_1|+|x-y_2|)^{2n}}\omega\big(\frac{|y_1-c_{1,k_1}|}{|x-y_1|+|x-y_2|}\big)dy_1dy_2dx\\
&\quad+CC_0\gamma^{1/2}\lambda^{-1/2}\sum \limits_{S_{1,l_1}}\sum \limits_{Q_{1,k_1}\subset S_{1,l_1}}|\lambda_{Q_{1,k_1}}|\sum\limits_{i=1}^{\infty}\int_{(S^*)^c\cap\mathscr{R}_{1, k_1}^i}\int_{Q_{1, k_1}^*}\int_{\reals^n}|b_1(x)-b_{1,Q_{1,k_1}}|\\ \notag
&\quad\times|b_2(y_2)-b_{2,Q_{1,k_1}}|\frac{|a_{Q_{1, k_1}}(y_1)||}{(|x-y_1|+|x-y_2|)^{2n}}\omega\big(\frac{|y_1-c_{1,k_1}|}{|x-y_1|+|x-y_2|}\big)dy_1dy_2dx\\
&=:|E_{2,3}^1|+|E_{2,3}^2|.
  \end{align*}
  For any $h\in \mathbb{N}$, if $y_2\in \mathscr{R}_{1, k_1}^h$, note that $y_1\in Q_{1,k_1}$, then \[|x-y_1|+|x-y_2|\geq|y_1-y_2|\sim |y_2-c_{1,k_1}|\sim l_{2^{h+2}Q_{1,k_1}}.\]
  On the other hand, for any $i\in \mathbb{N}$, if $x\in \mathscr{R}_{1, k_1}^i$ and $y_1\in Q_{1,k_1}$, then
  \begin{equation}\label{i221}|x-y_1|+|x-y_2|\geq|x-y_1|\sim l_{2^{i+2}Q_{1,k_1}}.\end{equation}
 By the geometric properties of $y_1$, $y_2$, $x$ above, we may obtain
   \begin{align}\label{i22}
   &|E_{2,3}^1|\\\notag &\leq CC_0\gamma^{1/2}\lambda^{-1/2}\sum \limits_{S_{1,l_1}}\sum \limits_{Q_{1,k_1}\subset S_{1,l_1}}|\lambda_{Q_{1,k_1}}|\sum\limits_{i=1}^{\infty}\sum\limits_{h=1}^{\infty}\int_{(S^*)^c\cap\mathscr{R}_{1, k_1}^i}\int_{\mathscr{R}_{1, k_1}^h}\int_{\reals^n}|b_1(x)-b_{1,Q_{1,k_1}}|\\ \notag
&\quad\times|b_2(y_2)-b_{2,Q_{1,k_1}}|\frac{|a_{Q_{1, k_1}}(y_1)|}{(|x-y_1|+|x-y_2|)^{2n}}\omega\big(\frac{|y_1-c_{1,k_1}|}{|x-y_1|+|x-y_2|}\big)dy_1dy_2dx\\ \notag
&\leq CC_0\gamma^{1/2}\lambda^{-1/2}\sum \limits_{S_{1,l_1}}\sum \limits_{Q_{1,k_1}\subset S_{1,l_1}}|\lambda_{Q_{1,k_1}}|\sum\limits_{i=1}^{\infty}\sum\limits_{h=1}^{\infty}\int_{(S^*)^c\cap\mathscr{R}_{1, k_1}^i}\int_{\mathscr{R}_{1, k_1}^h}\int_{\reals^n}|b_1(x)-b_{1,Q_{1,k_1}}|  \end{align}
\begin{align*}
\quad\times|b_2(y_2)-b_{2,Q_{1,k_1}}|\frac{|a_{Q_{1, k_1}}(y_1)|}{|2^{i+2}Q_{1,k_1}||2^{h+2}Q_{1,k_1}|}\omega(2^{-i})^{1/2}\omega(2^{-h})^{1/2}dy_1dy_2dx.
  \end{align*}
It is easy to see that
\begin{align}\label{esti y_2}
\sum\limits_{h=1}^{\infty}\omega(2^{-h})^{1/2}\int_{\mathscr{R}_{1, k_1}^h}\frac{|b_2(y_2)-b_{2,Q_{1,k_1}}|}{|2^{h+2}Q_{1,k_1}|}dy_2\leq C\sum\limits_{h=1}^{\infty}\omega(2^{-h})^{1/2}h\norm{b_2}_*\leq C.
\end{align}
Since $a(y_1)\in L^1(\reals^n)$, putting the above estimate into ($\ref{i22}$), we have
\begin{align*}
|E_{2,3}^1|&\leq CC_0\gamma^{1/2}\lambda^{-1/2}\sum \limits_{S_{1,l_1}}\sum \limits_{Q_{1,k_1}\subset S_{1,l_1}}|\lambda_{Q_{1,k_1}}|\sum\limits_{i=1}^{\infty}\omega(2^{-i})^{1/2}\int_{2^{i+2}Q_{1,k_1}}\frac{|b_1(x)-b_{1,Q_{1,k_1}}|}{|2^{i+2}Q_{1,k_1}|}dx\\
&\leq CC_0\gamma^{1/2}\lambda^{-1/2}\sum \limits_{S_{1,l_1}}\sum \limits_{Q_{1,k_1}\subset S_{1,l_1}}|\lambda_{Q_{1,k_1}}|\sum\limits_{i=1}^{\infty}\omega(2^{-i})^{1/2}i\norm{b_1}_*\\
&\leq CC_0\gamma^{1/2}\lambda^{-1/2}.
\end{align*}
  If $y_2\in Q_{1, k_1}^*$, note that $x\in(8\sqrt{n}Q_{1, k_1})^c$, then \[|x-y_1|+|x-y_2|\geq|x-y_2|\geq Cl_{Q_{1,k_1}}.\]
  By the definition of $|E_{2,3}^2|$ and ($\ref{i221}$), we have
\begin{align*}
|E_{2,3}^2|&\leq CC_0\gamma^{1/2}\lambda^{-1/2}\sum \limits_{S_{1,l_1}}\sum \limits_{Q_{1,k_1}\subset S_{1,l_1}}|\lambda_{Q_{1,k_1}}|\sum\limits_{i=1}^{\infty}\int_{(S^*)^c\cap\mathscr{R}_{1, k_1}^i}\int_{Q_{1, k_1}^*}\int_{\reals^n}|b_1(x)-b_{1,Q_{1,k_1}}|\\ \notag
&\quad\times|b_2(y_2)-b_{2,Q_{1,k_1}}|\frac{|a_{Q_{1, k_1}}(y_1)|}{|2^{i+2}Q_{1,k_1}||Q_{1, k_1}^*|}\omega(2^{-i})dy_1dy_2dx\\
&\leq CC_0\gamma^{1/2}\lambda^{-1/2}\sum \limits_{S_{1,l_1}}\sum \limits_{Q_{1,k_1}\subset S_{1,l_1}}|\lambda_{Q_{1,k_1}}|\sum\limits_{i=1}^{\infty}\int_{2^{i+2}Q_{1,k_1}}\int_{\reals^n}|b_1(x)-b_{1,Q_{1,k_1}}|\\ \notag
&\quad\times\frac{|a_{Q_{1, k_1}}(y_1)|}{|2^{i+2}Q_{1,k_1}|}\omega(2^{-i})dy_1dx\\
&\leq CC_0\gamma^{1/2}\lambda^{-1/2}\sum \limits_{S_{1,l_1}}\sum \limits_{Q_{1,k_1}\subset S_{1,l_1}}|\lambda_{Q_{1,k_1}}|\sum\limits_{i=1}^{\infty}\omega(2^{-i})\frac{1}{|2^{i+2}Q_{1,k_1}|}\\ \notag
&\quad\times\int_{2^{i+2}Q_{1,k_1}}|b_1(x)-b_{1,Q_{1,k_1}}|dx\\
&\leq CC_0\gamma^{1/2}\lambda^{-1/2}\sum \limits_{S_{1,l_1}}\sum \limits_{Q_{1,k_1}\subset S_{1,l_1}}|\lambda_{Q_{1,k_1}}|\sum\limits_{i=1}^{\infty}\omega(2^{-i})i\norm{b_1}_*\\
&\leq CC_0\gamma^{1/2}\lambda^{-1/2}.
  \end{align*}
  Hence, we have shown that
  \[|E_{2,3}|\leq|E_{2,3}^1|+|E_{2,3}^2|\leq CC_0\gamma^{1/2}\lambda^{-1/2}.\]
Now we begin to consider $|E_{2,4}|$. Similarly,
 \begin{align*}
|E_{2,4}|&\leq CC_0\gamma^{1/2}\lambda^{-1/2}\sum \limits_{S_{1,l_1}}\sum \limits_{Q_{1,k_1}\subset S_{1,l_1}}|\lambda_{Q_{1,k_1}}|\sum\limits_{i=1}^{\infty}\int_{(S^*)^c\cap\mathscr{R}_{1, k_1}^i}\iint_{(\reals^n)^2}|b_1(y_1)-b_{1,Q_{1,k_1}}|\\ \notag
&\quad\times|b_2(y_2)-b_{2,Q_{1,k_1}}|\frac{|a_{Q_{1, k_1}}(y_1)|}{(|x-y_1|+|x-y_2|)^{2n}}\omega\big(\frac{|y_1-c_{1,k_1}|}{|x-y_1|+|x-y_2|}\big)dy_1dy_2dx
\end{align*}
  Repeating the same steps as in the estimate of $|E_{2,3}|$, we have
 \begin{align*}
  |E_{2,4}|&\leq CC_0\gamma^{1/2}\lambda^{-1/2}\sum \limits_{S_{1,l_1}}\sum \limits_{Q_{1,k_1}\subset S_{1,l_1}}|\lambda_{Q_{1,k_1}}|\sum\limits_{i=1}^{\infty}\int_{(S^*)^c\cap\mathscr{R}_{1, k_1}^i}\int_{\cup_{i=1}^{\infty}\mathscr{R}_{1, k_1}^h}\int_{\reals^n}|b_1(y_1)-b_{1,Q_{1,k_1}}|\\ \notag
&\quad\times|b_2(y_2)-b_{2,Q_{1,k_1}}|\frac{|a_{Q_{1, k_1}}(y_1)|}{(|x-y_1|+|x-y_2|)^{2n}}\omega\big(\frac{|y_1-c_{1,k_1}|}{|x-y_1|+|x-y_2|}\big)dy_1dy_2dx\\
&\quad+CC_0\gamma^{1/2}\lambda^{-1/2}\sum \limits_{S_{1,l_1}}\sum \limits_{Q_{1,k_1}\subset S_{1,l_1}}|\lambda_{Q_{1,k_1}}|\sum\limits_{i=1}^{\infty}\int_{(S^*)^c\cap\mathscr{R}_{1, k_1}^i}\int_{Q_{1, k_1}^*}\int_{\reals^n}|b_1(y_1)-b_{1,Q_{1,k_1}}|\\ \notag
&\quad\times|b_2(y_2)-b_{2,Q_{1,k_1}}|\frac{|a_{Q_{1, k_1}}(y_1)|}{(|x-y_1|+|x-y_2|)^{2n}}\omega\big(\frac{|y_1-c_{1,k_1}|}{|x-y_1|+|x-y_2|}\big)dy_1dy_2dx\\
&=:|E_{2,4}^1|+|E_{2,4}^2|.
  \end{align*}
  By the definition of $|E_{2,4}^1|$, one may obtain
 \begin{align*}
|E_{2,4}^1|&\leq CC_0\gamma^{1/2}\lambda^{-1/2}\sum \limits_{S_{1,l_1}}\sum \limits_{Q_{1,k_1}\subset S_{1,l_1}}|\lambda_{Q_{1,k_1}}|\sum\limits_{i=1}^{\infty}\sum\limits_{h=1}^{\infty}\int_{(S^*)^c\cap\mathscr{R}_{1, k_1}^i}\int_{\mathscr{R}_{1, k_1}^h}\int_{\reals^n}|b_1(x)-b_{1,Q_{1,k_1}}|\\ \notag
&\quad\times|b_2(y_2)-b_{2,Q_{1,k_1}}|\frac{|a_{Q_{1, k_1}}(y_1)|}{|x-y_1|^n|2^{h+2}Q_{1,k_1}|}\omega\bigg(\frac{y_1-c_{1,k_1}}{|x-y_1|}\bigg)^{1/2}\omega(2^{-h})^{1/2}dy_1dy_2dx.
\end{align*}
By $(\ref{esti y_2})$, and integral for $x$ firstly, we have
\begin{align*}
 |E_{2,4}^1|
&\leq CC_0\gamma^{1/2}\lambda^{-1/2}\sum \limits_{S_{1,l_1}}\sum \limits_{Q_{1,k_1}\subset S_{1,l_1}}|\lambda_{Q_{1,k_1}}|\sum\limits_{i=1}^{\infty}\int_{\mathscr{R}_{1, k_1}^i}\int_{Q_{1,k_1}}\frac{|b_1(y_1)-b_{1,Q_{1,k_1}}|}
{|Q_{1,k_1}||x-y_1|^n}\\
&\quad\times\omega\bigg(\frac{y_1-c_{1,k_1}}{|x-y_1|}\bigg)^{1/2}dy_1dx\\
&\leq CC_0\gamma^{1/2}\lambda^{-1/2}\sum \limits_{S_{1,l_1}}\sum \limits_{Q_{1,k_1}\subset S_{1,l_1}}|\lambda_{Q_{1,k_1}}|\int_{Q_{1,k_1}}\frac{|b_1(y_1)-b_{1,Q_{1,k_1}}|}
{|Q_{1,k_1}|}dy_1\\
&\leq CC_0\gamma^{1/2}\lambda^{-1/2}\sum \limits_{S_{1,l_1}}\sum \limits_{Q_{1,k_1}\subset S_{1,l_1}}|\lambda_{Q_{1,k_1}}|\norm{b_1}_*\\
&\leq CC_0\gamma^{1/2}\lambda^{-1/2}.
 \end{align*}
 The estimate for $|E_{2,4}^2|$ is quite similar to $|E_{2,3}^2|$, we may get $|E_{2,4}^2|\leq CC_0\gamma^{1/2}\lambda^{-1/2}$.\\ \\
 \noindent
 \bull
 \textbf {Estimate for $|E_3|$.} Since $|E_3|$ is a symmetrical case of $|E_2|$, then
 \[|E_{3}|\leq CC_0\gamma^{1/2}\lambda^{-1/2}.\]
 \noindent
 \bull
 \textbf {Estimate for $|E_4|$.}
 \begin{align*}
T_{\Pi b}(h_1,h_2)&=[b_1,[b_2, T]_2,]_1(h_1,h_2)\\ \notag
&=\int_{(\reals^n)^m}\prod_{j=1}^{2}\big(b_j(x)-b_j(y_j)\big)K(x, y_1,y_2)h_1(y_1) h_2(y_2)dy_1dy_2\\
&=\sum \limits_{S_{1,l_1}}\sum \limits_{Q_{1,k_1}\subset S_{1,l_1}}\sum \limits_{S_{2,l_2}}\sum \limits_{Q_{2,k_2}\subset S_{2,l_2}}\lambda_{Q_{1,k_1}}\lambda_{Q_{2,k_2}}\big(b_1(x)-b_{1,Q_{1,k_1}}\big)\big(b_2(x)-b_{2,Q_{1,k_1}}\big)\\
&\quad\times T(a_{Q_{1,k_1}},a_{Q_{2,k_2}})(x)\\
&\quad-\sum \limits_{S_{1,l_1}}\sum \limits_{Q_{1,k_1}\subset S_{1,l_1}}\sum \limits_{S_{2,l_2}}\sum \limits_{Q_{2,k_2}\subset S_{2,l_2}}\lambda_{Q_{1,k_1}}\lambda_{Q_{2,k_2}}\big(b_2(x)-b_{2,Q_{1,k_1}}\big)\\&\quad\times T((b_1-b_{1,Q_{1,k_1}})a_{Q_{1,k_1}},a_{Q_{2,k_2}})(x)\\
&\quad-\sum \limits_{S_{1,l_1}}\sum \limits_{Q_{1,k_1}\subset S_{1,l_1}}\sum \limits_{S_{2,l_2}}\sum \limits_{Q_{2,k_2}\subset S_{2,l_2}}\lambda_{Q_{1,k_1}}\lambda_{Q_{2,k_2}}\big(b_1(x)-b_{1,Q_{1,k_1}}\big)\\&\quad\times T(a_{Q_{1,k_1}}, (b_2-b_{2,Q_{1,k_1}})a_{Q_{2,k_2}})(x)\\
&\quad+\sum \limits_{S_{1,l_1}}\sum \limits_{Q_{1,k_1}\subset S_{1,l_1}}\sum \limits_{S_{2,l_2}}\sum \limits_{Q_{2,k_2}\subset S_{2,l_2}}\lambda_{Q_{1,k_1}}\lambda_{Q_{2,k_2}}\\
&\quad\times T((b_1-b_{1,Q_{1,k_1}})a_{Q_{1,k_1}}, (b_2-b_{2,Q_{1,k_1}})a_{Q_{2,k_2}})(x)\\
&=:I_{4,1}(x)+I_{4,2}(x)+I_{4,3}(x)+I_{4,4}(x).
\end{align*}
Thus, we obtain
 \begin{align*}
 |E_4|&=|\{x\in\reals^n/S^*:|T_{\pi b}(h_1,h_2)(x)|>\lambda/4\}|\\
 &\leq |\{x\in\reals^n/S^*:|I_{4,1}(x)|>\lambda/16\}|+|\{x\in\reals^n/S^*:|I_{4,2}(x)|>\lambda/16\}|\\
 &\quad+|\{x\in\reals^n/S^*:|I_{4,3}(x)|>\lambda/16\}|+|\{x\in\reals^n/S^*:|I_{4,4}(x)|>\lambda/16\}|\\
 &=:|E_{4,1}|+|E_{4,2}|+|E_{4,3}|+|E_{4,4}|.
 \end{align*}
 Now we begin to consider $|E_{4,1}|$. By the definition of $I_{4,1}(x)$ , we can write
    \begin{align*}
 |I_{4,1}(x)|&\leq\sum \limits_{S_{1,l_1}}\sum \limits_{Q_{1,k_1}\subset S_{1,l_1}}\sum \limits_{S_{2,l_2}}\sum \limits_{Q_{2,k_2}\subset S_{2,l_2}}|\lambda_{Q_{1,k_1}}||\lambda_{Q_{2,k_2}}|\bigg|\iint_{(\reals^n)^2}\big(b_1(x)-b_{1,Q_{1,k_1}}\big)\\
&\quad\times\big(b_2(x)-b_{2,Q_{1,k_1}}\big)K(x,y_1,y_2)a_{Q_{1,k_1}}(y_1)
a_{Q_{2,k_2}}(y_2)dy_1dy_2\bigg|.
  \end{align*}

 Fix for a moment $k_1$, $k_2$ and assume, without loss of generality, that $l(Q_{1,k_1})\leq l(Q_{2,k_2})$. By the moment condition of $H^1$-atoms and the regularity condition $(\ref{k-H2s})$ of the kernel $K$, we have
  \begin{align*}
  &\bigg|\int_{\reals^n}K(x,y_1,y_2)a_{1,k_1}(y_1)dy_1\bigg|=
 \bigg|\int_{\reals^n}\big(K(x,y_1,y_2)-K(x,c_{1,k_1},y_2)\big)a_{1,k_1}(y_1)dy_1\bigg|\\
 &\leq \bigg|\int_{\reals^n}\frac{C_0}{(|x-y_1|+|x-y_2|)^{2n}}
 \omega\big(\frac{|y_1-c_{1,k_1}|}{|x-y_1|+|x-y_2|}
\big)a_{Q_{1,k_1}}(y_1)dy_1\bigg|.
  \end{align*}
  Recalling the definition of $\mathscr{R}_{1, k_1}^i$, $\mathscr{R}_{2, k_2}^h$, and note that $y_1\in Q_{1,k_1}$, $y_2\in Q_{2,k_2}$, it's obvious that, for any fixed $i, h, k_1, k_2$, if $x\in(S^*)^c\cap\mathscr{R}_{1, k_1}^i\cap\mathscr{R}_{2, k_2}^h$, then we have
  \[|x-y_1|\sim 2^il_{Q_{1,k_1}},\quad|x-y_2|\sim 2^hl_{Q_{2,k_2}}.\]
  This and the nondecreasing property of $\omega$ give
   \begin{align*}
   \frac{\omega\big(\frac{|y_1-c_{1,k_1}|}{|x-y_1|+|x-y_2|}\big)^{\frac{1}{2}}}{(|x-y_1|+|x-y_2|)^{n}}\leq
   \frac{\omega\big(\frac{l_{Q_{1,k_1}}}{|x-y_1|+|x-y_2|}\big)^{\frac{1}{2}}}{(|x-y_1|+|x-y_2|)^{n}}\lesssim
  \prod_{i=1}^2\frac{\omega\big(\frac{l_{Q_{i,k_i}}}{|x-y_i|}\big)^{\frac{1}{4}}}{|x-y_i|^{\frac{n}{2}}}\lesssim \frac{\omega(2^{-i})^{\frac{1}{4}}\omega(2^{-h})^{\frac{1}{4}}}{(2^il_{Q_{1,k_1}}2^hl_{Q_{2,k_2}})^{\frac{n}{2}}}. \end{align*}
  By $(\ref{key})$, Chebychev inequality and the estimate above, we control $|E_{4,1}|$ by
\begin{align}\label{i42}&\frac{CC_0}{\lambda}\sum \limits_{S_{1,l_1}}\sum \limits_{Q_{1,k_1}\subset S_{1,l_1}}\sum \limits_{S_{2,l_2}}\sum \limits_{Q_{2,k_2}\subset S_{2,l_2}}\sum\limits_{i=1}^{\infty}\sum\limits_{h=1}^{\infty}|\lambda_{Q_{1,k_1}}||\lambda_{Q_{2,k_2}}|\int_{(S^*)^c\cap\mathscr{R}_{1, k_1}^i\cap\mathscr{R}_{2, k_2}^h}\\ \notag
&\quad\times \iint_{(\reals^n)^2}|b_1(x)-b_{1,Q_{1,k_1}}||b_2(x)-b_{2,Q_{1,k_1}}|\frac{|a_{Q_{1,k_1}}(y_1)||a_{Q_{2,k_2}}(y_2)|}{(|x-y_1|+|x-y_2|)^{2n}}\\ \notag
&\quad\times\omega\big(\frac{|y_1-c_{1,k_1}|}{|x-y_1|+|x-y_2|}\big)dy_1dy_2dx\\ \notag
&\leq\frac{CC_0}{\lambda}\sum \limits_{S_{1,l_1}}\sum \limits_{Q_{1,k_1}\subset S_{1,l_1}}\sum \limits_{S_{2,l_2}}\sum \limits_{Q_{2,k_2}\subset S_{2,l_2}}\sum\limits_{i=1}^{\infty}\sum\limits_{h=1}^{\infty}\omega(2^{-i})^{\frac{1}{4}}\omega(2^{-h})^{\frac{1}{4}}|
\lambda_{Q_{1,k_1}}||\lambda_{Q_{2,k_2}}|\\ \notag
&\quad\times\int_{(S^*)^c\cap\mathscr{R}_{1, k_1}^i\cap\mathscr{R}_{2, k_2}^h}\frac{|b_1(x)-b_{1,Q_{1,k_1}}|}{(2^il_{Q_{1,k_1}}2^hl_{Q_{2,k_2}})^{\frac{n}{2}}}\times\bigg(\iint_{(\reals^n)^2}
|b_2(x)-b_{2,Q_{1,k_1}}|\\ \notag
&\quad\times\frac{|a_{Q_{1,k_1}}(y_1)||a_{Q_{2,k_2}}(y_2)|}{(|x-y_1|+|x-y_2|)^{n}}\omega\bigg(\frac{|y_1-c_{1,k_1}|}{|x-y_1|+|x-y_2|}
\bigg)^{\frac{1}{2}}dy_1dy_2\bigg)dx.
  \end{align}
Let's first consider the inside integrals, by the H\"{o}lder inequality, we may have
\begin{align}\label{i41}
&\int_{(S^*)^c\cap\mathscr{R}_{1, k_1}^i\cap\mathscr{R}_{2, k_2}^h}\frac{|b_1(x)-b_{1,Q_{1,k_1}}|}{(2^il_{Q_{1,k_1}}2^hl_{Q_{2,k_2}})^{\frac{n}{2}}}\times\bigg(\iint_{(\reals^n)^2}
|b_2(x)-b_{2,Q_{1,k_1}}|\\ \notag
&\quad\times\frac{|a_{Q_{1,k_1}}(y_1)||a_{Q_{2,k_2}}(y_2)|}{(|x-y_1|+|x-y_2|)^{n}}\omega\big(\frac{|y_1-c_{1,k_1}|}{|x-y_1|+|x-y_2|}
\big)^{\frac{1}{2}}dy_1dy_2\bigg)dx\\ \notag
&\leq\bigg(\frac{1}{(2^hl_{Q_{2,k_2}})^{n}}\int_{\mathscr{R}_{2, k_2}^h}|b_1(x)-b_{1,Q_{1,k_1}}|^2dx\bigg)^{\frac{1}{2}}\\ \notag
&\quad\times\bigg(\frac{1}{(2^il_{Q_{1,k_1}})^{n}}\int_{(S^*)^c\cap\mathscr{R}_{1, k_1}^i}\bigg|\iint_{(\reals^n)^2}
|b_2(x)-b_{2,Q_{1,k_1}}|\end{align}\begin{align*}\quad\times\frac{|a_{Q_{1,k_1}}(y_1)||a_{Q_{2,k_2}}(y_2)|}{(|x-y_1|+|x-y_2|)^{n}}\omega\big(\frac{|y_1-c_{1,k_1}|}{|x-y_1|+|x-y_2|}
\big)^{\frac{1}{2}}dy_1dy_2\bigg|^2dx\bigg)^\frac{1}{2}
\end{align*}
Note that $a_{2,k_2}(y_2)\in L^1(\reals^n)$, similar argument as in $(\ref{eq8})$ yields that
\begin{align*}
(\ref{i41})&\leq h^{\frac{1}{2}}\norm{b_2}_*^{\frac{1}{2}}\bigg[\frac{1}{(2^il_{Q_{1,k_1}})^{n}}\int_{(S^*)^c\cap\mathscr{R}_{1, k_1}^i}\big|\int_{\reals^n}
|b_2(x)-b_{2,Q_{1,k_1}}|\\ \notag
&\quad\times\sup\limits_{y_1, y_2\in S}\bigg(\frac{1}{(|x-y_1|+|x-y_2|)^{n}}\omega\big(\frac{|y_1-c_{1,k_1}|}{|x-y_1|+|x-y_2|}
\big)^{\frac{1}{2}}\bigg)|a_{Q_{1,k_1}}(y_1)|dy_1\big|^2dx\bigg]^\frac{1}{2}.
\end{align*}
Note that the integrals in the above inequality are independent of $S_{2,l_2}$ and $Q_{2,k_2}$ and $\omega$ is doubling, similarly as what we have done with $(\ref{rel1})$, for fixed $x\in (S^*)^c$ and any $y_1, y_2\in S$, we have
\begin{align}\label{rel2}
 &\sup\limits_{y_1, y_2\in S}\bigg(\frac{1}{(|x-y_1|+|x-y_2|)^{n}}\omega\big(\frac{|y_1-c_{1,k_1}|}{|x-y_1|+|x-y_2|}
\big)^{\frac{1}{2}}\bigg)\\ \notag
 &\approx\frac{1}{(|x-y_1|+|x-y_2|)^{n}}\omega\big(\frac{|y_1-c_{1,k_1}|}{|x-y_1|+|x-y_2|}
\big)^{\frac{1}{2}}.
 \end{align}
 Recalling $(I)$ in Theorem $\ref{thm1}$ and putting the inequality above into $(\ref{i42})$, we may get
\begin{align*}
|E_{4,1}|&\leq\frac{CC_0}{\lambda}\sum \limits_{S_{1,l_1}}\sum \limits_{Q_{1,k_1}\subset S_{1,l_1}}\sum\limits_{i=1}^{\infty}\sum\limits_{h=1}^{\infty}\omega(2^{-i})^{\frac{1}{4}}
\omega(2^{-h})^{\frac{1}{4}}h^{\frac{1}{2}}|\lambda_{Q_{1,k_1}}|
\bigg(\frac{1}{(2^il_{Q_{1,k_1}})^{n}}
\\&\quad\times\int_{(S^*)^c\cap\mathscr{R}_{1, k_1}^i}\bigg|\int_{\reals^n}
|b_2(x)-b_{2,Q_{1,k_1}}|\big(\sum \limits_{S_{2,l_2}}\sum \limits_{Q_{2,k_2}\subset S_{2,l_2}}|\lambda_{Q_{2,k_2}}|\big)\\ \notag
&\quad\times\sup\limits_{y_1, y_2\in S}\bigg(\frac{1}{(|x-y_1|+|x-y_2|)^{n}}\omega\big(\frac{|y_1-c_{1,k_1}|}{|x-y_1|+|x-y_2|}
\big)^{\frac{1}{2}}\bigg)|a_{Q_{1,k_1}}(y_1)|dy_1\bigg|^2dx\bigg)^\frac{1}{2}\\
&\leq CC_0\gamma^{\frac{1}{2}}\lambda^{-\frac{1}{2}}\sum \limits_{S_{1,l_1}}\sum \limits_{Q_{1,k_1}\subset S_{1,l_1}}\sum\limits_{i=1}^{\infty}\sum\limits_{h=1}^{\infty}\omega(2^{-i})^{\frac{1}{4}}
\omega(2^{-h})^{\frac{1}{4}}h^{\frac{1}{2}}|\lambda_{Q_{1,k_1}}|
\bigg(\frac{1}{(2^il_{Q_{1,k_1}})^{n}}\\
&\quad\times\int_{(S^*)^c\cap\mathscr{R}_{1, k_1}^i}\bigg|\int_{\reals^n}
|b_2(x)-b_{2,Q_{1,k_1}}|\bigg(\sum \limits_{S_{2,l_2}}\int_{S_{2,l_2}}\frac{1}{(|x-y_1|+|x-y_2|)^{n}}\\
&\quad\times\omega\big(\frac{|y_1-c_{1,k_1}|}{|x-y_1|+|x-y_2|}
\big)^{\frac{1}{2}} dy_2\bigg)|a_{Q_{1,k_1}}(y_1)|dy_1\bigg|^2dx\bigg)^\frac{1}{2}\\
&\leq CC_0\gamma^{\frac{1}{2}}\lambda^{-\frac{1}{2}}\sum \limits_{S_{1,l_1}}\sum \limits_{Q_{1,k_1}\subset S_{1,l_1}}\sum\limits_{i=1}^{\infty}\sum\limits_{h=1}^{\infty}\omega(2^{-i})^{\frac{1}{4}}
\omega(2^{-h})^{\frac{1}{4}}h^{\frac{1}{2}}|\lambda_{Q_{1,k_1}}|
\\
&\quad\times\bigg(\frac{1}{(2^il_{Q_{1,k_1}})^{n}}\int_{(S^*)^c\cap\mathscr{R}_{1, k_1}^i}|b_2(x)-b_{2,Q_{1,k_1}}|^2\big(\int_{\reals^n}
|a_{Q_{1,k_1}}(y_1)|dy_1\big)^2dx\bigg)^\frac{1}{2}\end{align*}\begin{align*}&\leq CC_0\gamma^{\frac{1}{2}}\lambda^{-\frac{1}{2}}\sum\limits_{i=1}^{\infty}\sum\limits_{h=1}^{\infty}\omega(2^{-i})^{\frac{1}{4}}
\omega(2^{-h})^{\frac{1}{4}}h^{\frac{1}{2}}i^{\frac{1}{2}}\\
&\leq CC_0\gamma^{\frac{1}{2}}\lambda^{-\frac{1}{2}}.
\end{align*}
Now we begin with the estimate for $|E_{4,2}|$.\par
Recalling the definition of $I_{4,2}(x)$, the moment condition of $H^1$-atoms and smoothness condition $(\ref{k-H2s})$. Similar to the estimates in $(\ref{i42})$, we may obtain
 \begin{align}\label{iE42}
 |E_{4,2}|&\leq\frac{CC_0}{\lambda}\sum \limits_{S_{1,l_1}}\sum \limits_{Q_{1,k_1}\subset S_{1,l_1}}\sum \limits_{S_{2,l_2}}\sum \limits_{Q_{2,k_2}\subset S_{2,l_2}}\sum\limits_{i=1}^{\infty}|\lambda_{Q_{1,k_1}}||\lambda_{Q_{2,k_2}}|\int_{(S^*)^c\cap\mathscr{R}_{1, k_1}^i}\iint_{(\reals^n)^2}\\ \notag
&\quad|b_1(x)-b_{1,Q_{1,k_1}}||b_2(y_2)-b_{2,Q_{1,k_1}}|\frac{|a_{Q_{1,k_1}}(y_1)||a_{Q_{2,k_2}}(y_2)|}{|x-y_1|^{n}}\\ \notag
&\quad\times\frac{1}{(|x-y_1|+|x-y_2|)^{n}}\omega\big(\frac{|y_1-c_{1,k_1}|}{|x-y_1|+|x-y_2|}\big)dy_1dy_2dx.
  \end{align}
First, we consider the following summation.
  \begin{equation}\label{ieq4}
  \sum \limits_{S_{2,l_2}}\sum \limits_{Q_{2,k_2}\subset S_{2,l_2}}\int_{\reals^n}|b_2(y_2)-b_{2,Q_{1,k_1}}|\frac{|\lambda_{Q_{2,k_2}}||a_{Q_{2,k_2}}(y_2)|}{(|x-y_1|+|x-y_2|)^{n}}
  \omega\big(\frac{|y_1-c_{1,k_1}|}{|x-y_1|+|x-y_2|}\big)dy_2.
  \end{equation}
Property $(I)$ in Theorem $\ref{thm1}$, inequality $(\ref{rel2})$, and size condition of $H^1$-atoms, that is, $\norm{a_{Q_{2,k_2}}}_{L^{\infty}}\leq |Q_{2,k_2}|^{-1}$, together with the H\"{o}lder inequality, enable us to obtain
  \begin{align*}
(\ref{ieq4})&\leq\sum \limits_{S_{2,l_2}}\sum \limits_{Q_{2,k_2}\subset S_{2,l_2}}|\lambda_{Q_{2,k_2}}|\bigg(\int_{\reals^n}|b_2(y_2)-b_{2,Q_{1,k_1}}|^2|a_{Q_{2,k_2}}(y_2)|dy_2\bigg)^{\frac{1}{2}}\\
&\quad\times\bigg(\int_{\reals^n}\frac{1}{(|x-y_1|+|x-y_2|)^{2n}}
  \omega\big(\frac{|y_1-c_{1,k_1}|}{|x-y_1|+|x-y_2|}\big)^2|a_{Q_{2,k_2}}(y_2)|dy_2\bigg)^{\frac{1}{2}}\\
  &\leq\omega(2^{-i})\sum \limits_{S_{2,l_2}}\sum \limits_{Q_{2,k_2}\subset S_{2,l_2}}|\lambda_{Q_{2,k_2}}|\norm{b_2}_*^{\frac{1}{2}}\sup\limits_{y_1, y_2\in S}\bigg(\frac{1}{(|x-y_1|+|x-y_2|)^{n}}
  \\
  &\quad\times\omega\big(\frac{|y_1-c_{1,k_1}|}{|x-y_1|+|x-y_2|}\big)^{\frac{1}{2}}\bigg)\\
  &\leq C(\gamma\lambda)^{\frac{1}{2}}\omega(2^{-i})^{\frac{1}{2}}\sum \limits_{S_{2,l_2}}\int_{S_{2,l_2}}\frac{1}{(|x-y_1|+|x-y_2|)^{n}}
  \omega\big(\frac{|y_1-c_{1,k_1}|}{|x-y_1|+|x-y_2|}\big)^{\frac{1}{2}}dy_2\\
  &\leq C(\gamma\lambda)^{\frac{1}{2}}\omega(2^{-i})^{\frac{1}{2}}.
  \end{align*}
 Therefore, by $(\ref{iE42})$ and note that $a_{Q_{1,k_1}}(y_2)\in L^1(\reals^n)$, we have
  \begin{align*}\label{iE42}
 |E_{4,2}|&\leq CC_0\gamma^{\frac{1}{2}}\lambda^{-\frac{1}{2}}\sum \limits_{S_{1,l_1}}\sum \limits_{Q_{1,k_1}\subset S_{1,l_1}}\sum\limits_{i=1}^{\infty}\omega(2^{-i})^{\frac{1}{2}}|\lambda_{Q_{1,k_1}}|\int_{(S^*)^c\cap\mathscr{R}_{1, k_1}^i}\int_{\reals^n}\frac{1}{|x-y_1|^2}\\ \notag
&\quad\times|b_1(x)-b_{1,Q_{1,k_1}}||a_{Q_{1,k_1}}(y_1)|dy_1dx\\
&\leq CC_0\norm{b_1}_*\gamma^{\frac{1}{2}}\lambda^{-\frac{1}{2}}\sum \limits_{S_{1,l_1}}\sum \limits_{Q_{1,k_1}\subset S_{1,l_1}}|\lambda_{Q_{1,k_1}}|\sum\limits_{i=1}^{\infty}\omega(2^{-i})^{\frac{1}{2}}i^{\frac{1}{2}}\leq CC_0\gamma^{\frac{1}{2}}\lambda^{-\frac{1}{2}}.
  \end{align*}
Since $|E_{4,3}|$ is a symmetrical case of $|E_{4,2}|$ we may also obtain
  \[|E_{4,3}|\leq CC_0\gamma^{\frac{1}{2}}\lambda^{-\frac{1}{2}}.\]
 Similar argument still works as in $(\ref{eq5})$, we may have
   \[|E_{4,4}|\leq CC_2^{\frac{1}{2}}\lambda^{-\frac{1}{2}}.\]
This completes the estimate for $|E_4|$.
Thus, we have proved inequality (\ref{lasti2i}) and the proof of Theorem $\ref{thm1itrated}$ is finished.
\end{proof}

\end{document}